\documentclass[aop,preprint]{imsart}

\RequirePackage[OT1]{fontenc}
\RequirePackage{amsthm,amsmath}
\usepackage{amsfonts}
\RequirePackage[numbers]{natbib}
\RequirePackage[colorlinks,citecolor=blue,urlcolor=blue]{hyperref}

\usepackage{graphicx}
\usepackage{tikz}
\usepackage{array}
\usetikzlibrary{calc, intersections}
\usetikzlibrary{decorations.markings}
\usetikzlibrary{decorations.pathreplacing}
\input xy
\xyoption{all}

\arxiv{arXiv:1605.05819}

\startlocaldefs
\numberwithin{equation}{section}

\theoremstyle{plain}
\newtheorem{theorem}{Theorem}[section]
\newtheorem{lemma}[theorem]{Lemma}
\newtheorem{proposition}[theorem]{Proposition}
\newtheorem{corollary}[theorem]{Corollary}

\theoremstyle{definition}
\newtheorem{definition}[theorem]{Definition}

\newtheorem{notation}[theorem]{Notation}

\theoremstyle{remark}
\newtheorem{remark}[theorem]{Remark}
\newtheorem{example}[theorem]{Example}

\newtheorem{assumption}[theorem]{Assumption}
\endlocaldefs

\newcommand{\grad}{{\mathrm{grad}}}
\newcommand{\diag}{{\mathrm{diag}}}
\newcommand{\Ric}{{\mathrm{Ric}}}
\newcommand{\ppi}{\boldsymbol{\pi}}
\newcommand{\Hess}{{\mathrm{Hess}} \nobreak\hspace{.16667em plus .08333em} }

\begin{document}

\begin{frontmatter}
\title{Exponentially concave functions and a new information geometry}
\runtitle{Exponentially concave functions}

\begin{aug}
\author{\fnms{Soumik} \snm{Pal}\thanksref{t2,m1}\ead[label=e1]{soumikpal@gmail.com}}
\and
\author{\fnms{Ting-Kam Leonard} \snm{Wong}\thanksref{t2,m2}\ead[label=e2]{tkleonardwong@gmail.com}}

\thankstext{t2}{This research is partially supported by NSF grants DMS-1308340 and DMS-1612483.}
\runauthor{S.~Pal and T.-K.~L.~Wong}

\affiliation{University of Washington\thanksmark{m1} and University of Southern California\thanksmark{m2}}

\address{University of Washington\\C-547 Padelford Hall\\Seattle, Washington 98195\\USA\\
\printead{e1}}

\address{University of Southern California\\406H Kaprielian Hall\\Los Angeles, California 90089\\USA\\
\printead{e2}}
\end{aug}

\begin{abstract}
A function is exponentially concave if its exponential is concave. We consider exponentially concave functions on the unit simplex. In a previous paper we showed that gradient maps of exponentially concave functions provide solutions to a Monge-Kantorovich optimal transport problem and give a better gradient approximation than those of ordinary concave functions. The approximation error, called L-divergence, is different from the usual Bregman divergence.
 Using tools of information geometry and optimal transport, we show that L-divergence induces a new information geometry on the simplex consisting of a Riemannian metric and a pair of dually coupled affine connections which defines two kinds of geodesics. We show that the induced geometry is dually projectively flat but not flat. Nevertheless, we prove an analogue of the celebrated generalized Pythagorean theorem from classical information geometry. On the other hand, we consider displacement interpolation under a Lagrangian integral action that is consistent with the optimal transport problem and show that the action minimizing curves are dual geodesics. The Pythagorean theorem is also shown to have an interesting application of determining the optimal trading frequency in stochastic portfolio theory.
\end{abstract}

\begin{keyword}[class=MSC]
\kwd[Primary ]{60E05}
\kwd[; secondary ]{52A41}
\end{keyword}

\begin{keyword}
\kwd{information geometry}
\kwd{optimal transport}
\kwd{exponential concavity}
\kwd{L-divergence}
\kwd{generalized Pythagorean theorem}
\kwd{functionally generated portfolio}
\kwd{stochastic portfolio theory}
\end{keyword}

\end{frontmatter}


\section{Introduction} \label{sec:intro}

\begin{definition} [Exponential concavity] \label{def:exp.concave}
Let $\Omega \subset {\mathbb{R}}^d$ be convex. We say that a function $\varphi: \Omega \rightarrow {\mathbb{R}} \cup \{-\infty\}$ is exponentially concave if $\Phi = e^{\varphi}$ is concave. (By convention we set $e^{-\infty} = 0$.)
\end{definition}

Throughout this paper we let $\Omega$ be the open unit simplex
\begin{equation} \label{eqn:simplex}
\Delta_n = \left\{ p = (p_1, \ldots, p_n) \in {\mathbb{R}}^n: p_i > 0, \sum_{i = 1}^n p_i = 1 \right\},
\end{equation}
regarded as the collection of strictly positive probability distributions on a set with $n$ elements. This is due to the applications we have in mind, although many generalizations are possible. An interesting property of exponentially concave functions is that their gradient maps give a better first-order approximation those of than ordinary concave functions. In \cite{PW14}, we introduced the concept of L-divergence. Let $\varphi: \Delta_n \rightarrow {\mathbb{R}}$ be a differentiable exponentially concave function. For $p, q \in \Delta_n$, concavity of $e^{\varphi}$ implies that 
\begin{equation} \label{eq:loglinear}
\varphi(p) + \log\left( 1+ \nabla \varphi(p) \cdot (q - p) \right) \geq \varphi(q),
\end{equation}
where $\nabla \varphi$ is the Euclidean gradient. Clearly this approximation is sharper than the linear approximation of $\varphi$ itself. The L-divergence of $\varphi$ is the error in this approximation:
\begin{equation} \label{eqn:L.divergence}
T\left(q \mid p\right) := \log \left(1 + \nabla \varphi(p) \cdot (q - p) \right) - \left(\varphi(q) - \varphi(p)\right) \geq 0.
\end{equation}

The extra concavity of exponentially concave functions have found several recent applications in analysis, probability and optimization. For example, in \cite{EKS15}, the equivalence of entropic curvature-dimension conditions and Bochner's inequality on metric measure spaces is established using the notion of $(K,N)$ convexity. When $K \geq 0$ and $N < \infty$, the negative of a $(K,N)$ convex function is exponentially concave. Better gradient approximation has also led to better algorithms in optimization and machine learning such as those in \cite{HAK07, JRP08, MZJ15}, although the authors tend to replace the logarithmic term in \eqref{eq:loglinear} by a quadratic approximation. 

One of our primary applications in mind is related to stochastic portfolio theory. In \cite{F02, F99}, the author considers the gradient map of an exponentially concave function as a map from $\Delta_n$ to its closure $\overline{\Delta}_n$. The following restatement can be found in \cite[Proposition 6]{PW14}. Let $\varphi$ be a differentiable exponentially concave function on $\Delta_n$. For $p\in \Delta_n$, define $\ppi(p)\in {\mathbb{R}}^n$ by  
\begin{equation} \label{eqn:fgp.weights}
\ppi_i(p) = p_i \left(1 + \nabla \varphi(p) \cdot (e^{(i)} - p)\right), \quad i = 1, \ldots, n,
\end{equation}
where $e^{(i)} = (0, \ldots, 1, \ldots, 0)$ is the $i$th standard basis of ${\mathbb{R}}^n$. Then, it can be shown that $\ppi: \Delta_n \rightarrow \overline{\Delta}_n$. In keeping with standard definitions in the subject we will call this map the portfolio map. In this vein also see articles \cite{BF08, F01, FKK05, KR16, P16, S14, W14, W15}.

The L-divergence of $\varphi$ should be distinguished from the Bregman divergence of $\varphi$ defined by
\begin{equation} \label{eqn:Bregman.divergence}
D\left(q \mid p\right) := \nabla \varphi(p) \cdot (q - p) - \left(\varphi(q) - \varphi(p)\right).
\end{equation}
Bregman divergence was introduced in \cite{B67} and is widely applied in statistics and optimization. To see the difference consider two fundamental examples. For $q, p\in \Delta_n$, the Kullback-Leibler divergence (also known as relative entropy) is given by 
\[
H\left(q \mid p\right) = \sum_{i = 1}^n q_i \log \frac{q_i}{p_i}.   
\]
It can be shown that the relative entropy is the Bregman divergence of the Shannon entropy $H(p) = -\sum_{i = 1}^n p_i \log p_i$. On the other hand, fix $\pi \in \Delta_n$ and consider the cross entropy $\varphi(p)= \sum_{i=1}^n \pi_i \log p_i$. This is an exponentially concave function of $p$ whose associated portfolio map \eqref{eqn:fgp.weights} is constant: $\ppi(p) \equiv \pi$. The corresponding L-divergence is given by 
\begin{equation} \label{eqn:excess.growth}
T\left( q\mid p \right) = \log\left( \sum_{i=1}^n \pi_i \frac{q_i}{p_i}  \right)- \sum_{i=1}^n \pi_i \log\left( \frac{q_i}{p_i} \right).
\end{equation}
This quantity is sometimes referred to as the free energy in statistical physics. In finance it is called the diversification return in \cite{BF92, CZ14, EH06, W11}, the excess growth rate in \cite{F02, FS82, PW13}, the rebalancing premium in \cite{BNPS12}, and the volatility return in \cite{H14}.

In \cite{PW14} we introduced a Monge-Kantorovich optimal transport problem on ${\mathbb{R}}^{n-1}$ which can be solved using exponentially concave functions on the unit simplex. The cost function is defined for $(\theta, \phi)\in {\mathbb{R}}^{n-1} \times {\mathbb{R}}^{n-1}$ by
\begin{equation} \label{eqn:psi.intro}
c(\theta, \phi):=\psi(\theta-\phi), \quad \text{where}\quad \psi(x) := \log \left(1 + \sum_{i = 1}^{n-1} e^{x_i}\right)
\end{equation}
is strictly convex on ${\mathbb{R}}^{n-1}$. We will recall the details of this transport problem in Section \ref{sec:exponential.coordinate} and its relationship to exponentially concave functions. It suffices to say for now that, given a pair of Borel probability measures $P$ and $Q$ on ${\mathbb{R}}^{n-1}$ the optimal coupling of the two with respect to the above cost can be expressed in terms of the portfolio map of an exponentially concave function on the simplex. A related cost function appears in \cite{O07} in the completely different context of finding polytopes with given geometric data. It also appears to be related to the study of moment measures as introduced in \cite{CEK15} (see page 3836 in particular).  

\subsection{Our contributions}
In this paper we show that information geometry provides an elegant geometric structure underlying exponential concavity, L-divergence and the optimal transport problem. Here is a motivating question which is the starting point of this work. Suppose $\varphi$ is an exponentially concave function on $\Delta_n$ with its associated L-divergence $T(\cdot\mid \cdot)$. Can we geometrically characterize triplets $(p, q, r) \in \left(\Delta_n\right)^3$ such that $T(q\mid p)+ T(r\mid q)\le T(r\mid p)$? The answer to this question determines the optimal frequency of rebalancing the portfolio generated by $\varphi$ (see Section \ref{sec:application.fgp}). Also see Section \ref{sec:pqr.transport} for a transport interpretation of this inequality. 

Using tools of information geometry, we show that exponentially concave functions on $\Delta_n$ and their L-divergences induce a new geometric structure on the simplex $\Delta_n$ regarded as a smooth manifold of probability distributions. Let $\varphi$ be an exponentially concave function on $\Delta_n$. We only require that $\varphi$ is smooth and the Euclidean Hessian of $e^{\varphi}$ is strictly positive definite everywhere (see Assumption \ref{ass:regularity}). The induced geometric structure consists of a Riemannian metric $g$ and a dual pair of affine connections $\nabla$ and $\nabla^*$. These connections define via parallel transports two kinds of geodesic curves on $\Delta_n$ called primal and dual geodesics. Interestingly, the duality in this geometry goes hand in hand with the duality in the related Monge-Kantorovich optimal transport problem, and this work is the first which exploits this connection. We summarize our main results as follows. First we give the answer of the motivating question.

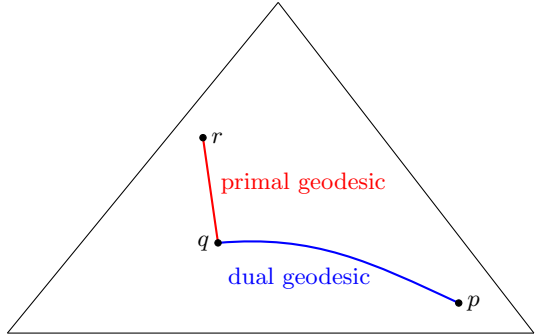
\begin{figure}[t]
\begin{tikzpicture}[scale = 0.4]
\draw (3.5, -2) to (-14, -2);
\draw (3.5, -2) to (-5, 9);
\draw (-14, -2) to (-5, 9);

\draw[thick, blue] (1, -1) to[out=155, in=5] (-7, 1);  
\draw[thick, red] (-7, 1) -- (-7.5, 4.5);  

\draw[black, fill] (1, -1) circle (3pt);
\node [right] at (1, -1) {$p$};
\draw[black, fill] (-7, 1) circle (3pt);
\node [left] at (-7, 1) {$q$};
\draw[black, fill] (-7.5, 4.5) circle (3pt);
\node [right] at (-7.5, 4.5) {$r$};

\node [blue, above] at (-4.3, -0.8) {{\footnotesize dual geodesic}};
\node [red, right] at (-7.2, 3) {{\footnotesize primal geodesic}};

\end{tikzpicture}
\caption{Generalized Pythagorean theorem for L-divergence.} \label{fig:pyth}
\end{figure}

\begin{theorem} [Generalized Pythagorean theorem] \label{thm:pyth}
Given $(p, q, r) \in \left(\Delta_n\right)^3$, consider the dual geodesic joining $q$ and $p$ and the primal geodesic joining $q$ and $r$. Consider the Riemannian angle between the geodesics at $q$. (See Proposition \ref{prop:metric} which expresses the Riemannian metric $g$ as a normalized Euclidean Hessian of $\Phi := e^{\varphi}$.) Then the difference
\begin{equation} \label{eqn:pyth}
T\left(q \mid p\right) + T\left(r \mid q\right) - T\left(r \mid p\right)
\end{equation}
is positive, zero or negative depending on whether the angle is less than, equal to, or greater than $90$ degrees (see Figure \ref{fig:pyth}). 
\end{theorem}

We also prove other remarkable properties of the geodesics: (i) There exist explicit coordinate systems under which the primal and dual geodesics are time changes of Euclidean straight lines (Theorem \ref{thm:geodesics}). In other words, the new geometry is dually projectively flat. In particular, the primal geodesics are Euclidean straight lines up to time reparameterization. Moreover, the primal and dual connections have constant sectional curvature $-1$ with respect to the Riemannian metric, and thus satisfy an Einstein condition (Corollary \ref{cor:Ricci}). The primal and dual geodesics can also be constructed as time changes of Riemannian gradient flows for the functions $T\left(r \mid \cdot\right)$ and $T\left( \cdot \mid p\right)$ (Theorem \ref{thm:gradient.flows}). This is remarkable because while the geodesic equations depend only on the local properties of $T\left(\xi \mid \xi'\right)$ near $\xi = \xi'$, the gradient flows are global as they involve the derivatives of $T\left(r \mid \cdot\right)$ and $T\left(\cdot \mid p\right)$. Indeed, this relation is known only for limited families of divergence including Bregman divergence and $\alpha$-divergence \cite{AA15}.

As shown in \cite[Chapter 1]{A16}, the generalized Pythagorean theorem holds for any Bregman divergence which induces a dually flat geometry. We will prove that the resulting geometries from L-divergences are not flat for $n \geq 3$ (Theorem \ref{thm:RC.curvature}). While some extensions of the generalized Pythagorean theorem hold in certain non-flat spaces (see for example \cite[Theorem 4.5]{A16}), they involve some extra terms. To the best of our knowledge Theorem \ref{thm:pyth} is the first exact Pythagorean theorem that holds in a geometry which is not dually flat. The difference \eqref{eqn:pyth} can also be given an optimal transport interpretation (Section \ref{sec:pqr.transport}).

(ii) We extend the static transport problem \eqref{eqn:psi.intro} to a time-dependent transport problem with a corresponding convex Lagrangian action. In Theorem \ref{thm:displacement.interpolation} we show that the action minimizing curves are the (reparametrized) dual geodesics which, in addition, satisfy the intermediate time optimality condition. This allows for a consistent displacement interpolation formulation between probability measures on the unit simplex. Previously, such studies focused almost exclusively on the Wasserstein spaces corresponding to the cost functions $C(x, y) = d(x, y)^\alpha$ (here $d$ is a metric on a Polish space with suitable properties and $\alpha \ge 1$). Displacement interpolation and the related concept of displacement convexity were introduced in \cite{MC97} and in the thesis \cite{MCThesis}. These ideas have grown to be immensely important in classical Wasserstein transport with fundamental implications in geometry, physics, probability and PDE. See \cite[Chapter 7]{V08} for a thorough discussion. Our Lagrangian, although convex, is not superlinear, and, therefore, is not covered by the standard theory. However, we expect it to lead to many equally remarkable properties.  

These results suggest plenty of problems for further research. Generalizing Theorem \ref{thm:pyth} to more than three points is of interest in stochastic portfolio theory.  Displacement interpolation has become an extremely important topic in optimal transport theory. Extensions to Riemanninan manifolds, done in \cite{CEMCS01}, have led to new functional inequalities. In another vein, \cite{LV09} defines Ricci curvatures on metric measure spaces in terms of displacement interpolation and displacement convexity. We expect that the displacement interpolation in this paper will lead to a new Otto calculus (\cite[Chapter 15]{V08}) and related PDEs (such as Hamilton-Jacobi equations). It appears that Bregman divergence and L-divergence are only two of an entire family of divergences with special properties and corresponding optimal transport problems. For example, see \cite{P17} which extends the optimal transport problem \eqref{eqn:psi.intro} via the cumulant generating function of a general probability distribution. We also believe that this new information geometry will be useful in dynamic optimization problems where the objective function is multiplicative in time. Finally, it is naturally of interest to study exponential concavity on general convex domains. 

\subsection{Related literature} 
We have mentioned the L-divergence and the Bregman divergence. In general, a divergence on a set (usually a manifold of probability distributions) is a non-negative function $D\left(q \mid p\right)$ such that $D\left(q \mid p\right) = 0$ if and only if $p = q$. Divergences are not metrics in general since they may be asymmetric and may not satisfy the triangle inequality. Apart from Bregman divergence, many families of divergences (such as $\alpha$-divergence and $f$-divergence) have been applied in information theory, statistics and other areas; see the survey \cite{B13} for a catalog of these divergences. Among these divergences, Bregman divergence plays a special role because it induces a dually flat geometry on the underlying space. First studied in the context of exponential families in statistical inference by \cite{NA82}, it gave rise to information geometry -- the geometric study of manifolds of probability distributions. Furthermore, Bregman divergence enjoys properties such as the generalized Pythagorean theorem and projection theorem which led to numerous applications. See \cite{A16,  AN07, CU14, KV11, M93} for introductions to this beautiful theory. The related concept of dual affine connection is also useful in affine differential geometry (see \cite{DNV90, K90, S07}). In \cite{M99} dually projectively flat manifolds are characterized in terms of the Bartlett tensors and conformal flatness. Here we identify a new and important class of examples and show that they have concrete applications. 

This work is motivated by our study in mathematical finance. Recently optimal transport has been applied to financial problems such as robust asset pricing; see, for example, \cite{ABPS13, BHP13, DS14}. This line of work has a somewhat different flavor than ours although they share the same goal: development of model-free mathematical finance. Portfolios generated by exponentially concave functions generate profit due to fluctuations of a sequence in $\Delta_n$ representing the stock market. This idea is sometimes called volatility harvesting and leads naturally to the transport problem \eqref{eqn:psi.intro}, as shown in \cite{PW14}. In this philosophy, our work can be interpreted as developing a notion of model-free volatility.

\subsection{Outline of the paper}
In the next section we recall the optimal transport problem formulated in \cite{PW14} using the exponential coordinate system. Its relation with functionally generated portfolio is also reviewed. In Section \ref{sec:duality} we relate exponential concavity with $c$-concavity and give a transport-motivated definition of L-divergence. Here duality plays a crucial role. After reviewing the basic concepts of information geometry, we derive in Section \ref{sec:geometry.computation} the geometric structure induced by an exponentially concave function. The properties of this new geometry are then studied in Section \ref{sec:geodesic}. In particular, we characterize the primal and dual geodesics and prove the generalized Pythagorean theorem which has an interesting application in mathematical finance. Finally, in Section \ref{sec:interpolation} we apply the geometric structure to construct a displacement interpolation for the associated optimal transport problem. Some technical and computational details are gathered in the Appendix.

\section{Optimal transport and portfolio maps} \label{sec:optimal.transport}
In this section we recall the optimal transport problem in \cite{PW14} using the exponential coordinate system. We also review the definition of functionally generated portfolio and explain how it relates to the transport problem.

\subsection{Exponential coordinate system} \label{sec:exponential.coordinate}
The exponential coordinate system defines a global coordinate system on $\Delta_n$ regarded as an $(n - 1)$-dimensional smooth manifold \cite[Section 2.2]{A16}. 

\begin{definition}[Exponential coordinate system]
The exponential coordinate $\theta = (\theta_1, \ldots, \theta_{n-1}) \in {\mathbb{R}}^{n-1}$ of $p \in \Delta_n$ is given by
\begin{equation} \label{eqn:exp.coordinate}
\theta_i = \log \frac{p_i}{p_n}, \quad i = 1, \ldots, n - 1.
\end{equation}
We denote this map by $\boldsymbol{\theta} : \Delta_n \rightarrow {\mathbb{R}}^{n-1}$. By convention we set $\theta_n \equiv 0$. The inverse transformation ${\bf p} := \boldsymbol{\theta}^{-1}$ is given by
\begin{equation} \label{eqn:theta.to.p}
p_i = {\bf p}_i(\theta) = e^{\theta_i - \psi(\theta)}, \quad 1 \leq i \leq n,
\end{equation}
where $\psi(\theta) = \log \left(1 + \sum_{i = 1}^{n-1} e^{\theta_i}\right) = \log \left(\sum_{i = 1}^n  e^{\theta_i}\right)$ as defined in \eqref{eqn:psi.intro}.
\end{definition}

The exponential coordinate system is the first of several coordinate systems we will introduce on the simplex. By changing coordinate systems, any function on $\Delta_n$ can be expressed as a function on ${\mathbb{R}}^{n-1}$ and vice versa. Explicitly, a function $\varphi$ on $\Delta_n$ can be expressed in exponential coordinates by $\theta \mapsto \varphi({\bf p}(\theta))$. To simplify the notations, we simply write $\varphi(p)$ or $\varphi(\theta)$ depending on the coordinate system used. For example, if $\varphi(p) = \sum_{i = 1}^n \pi_i \log p_i$ is the cross entropy where $\pi \in \Delta_n$, then $\varphi(\theta) = \sum_{i = 1}^{n-1} \pi_i \theta_i - \psi(\theta)$.

\subsection{The transport problem} \label{sec:the.transport.problem} We refer the reader to the books \cite{AG13, V08} for introductions to optimal transport and its interplay with analysis, probability and geometry. Let ${\mathcal{X}} = {\mathcal{Y}} = {\mathbb{R}}^{n-1}$ be equipped with the standard Euclidean metric and topology.  Let $P$ and $Q$ be Borel probability measures on ${\mathcal{X}}$ and ${\mathcal{Y}}$ respectively. By a coupling of $P$ and $Q$ we mean a Borel probability measure $R$ on ${\mathcal{X}} \times {\mathcal{Y}}$ whose marginals are $P$ and $Q$ respectively. Let $\Pi(P, Q)$ be the set of all couplings of $P$ and $Q$. This set is always non-empty as it contains the product measure $P \otimes Q$.

Given $P$ and $Q$ we consider the Monge-Kantorovich optimal transport problem with cost $c$ defined by \eqref{eqn:psi.intro}:
\begin{equation} \label{eqn:transport.problem}
\inf_{R \in \Pi(P, Q)} {\Bbb E} \left[ c(\theta, \phi)\right].
\end{equation}
Here the expectation is taken under the probability measure $R$ under which the random element $(\theta, \phi)$ has distribution $R$. If an optimal coupling takes the form $(\theta, F(\theta))$ for some measurable map $F: {\mathcal{X}} \rightarrow {\mathcal{Y}}$, we say that $F$ is a Monge transport map. 

In general, we may consider the optimal transport problem \eqref{eqn:transport.problem} with $c$ replaced by a general cost function denoted by $C(\cdot, \cdot)$ and ${\mathcal{X}}$, ${\mathcal{Y}}$ are general Polish spaces. The classical example is where ${\mathcal{X}} = {\mathcal{Y}}$ and $C$ is a power of the underlying metric: $C(x, y) = d(x, y)^{\alpha}$ (especially $\alpha = 1$ and $2$). For these costs rich and delicate theories have been developed on Euclidean spaces, Riemannian manifolds and geodesic metric measure spaces. However, we consider the cost function $c$ defined by \eqref{eqn:psi.intro}. 

\begin{remark}
The cost function
\[
\widetilde{c}(\theta, \phi) := \log\left(\frac{1}{n} + \frac{1}{n}\sum_{i = 1}^{n-1} e^{\theta_i - \phi_i}\right) - \sum_{i = 1}^{n-1} \frac{1}{n} (\theta_i - \phi_i)
\]
differs from $c(\cdot, \cdot)$ in a linear term which plays no role in optimal transport. Thus we may consider $\widetilde{c}$ instead. The advantage of $\widetilde{c}$ is that it is non-negative and, by Jensen's inequality, equals zero if and only if $\theta = \phi$. To be consistent with the notations in \cite{PW14} we will use the cost function $c$ in this paper.
\end{remark}

\begin{definition} [$c$-cyclical monotonicity]
A non-empty subset $A \subset {\mathcal{X}} \times {\mathcal{Y}}$ is $c$-cyclical monotone if and only if it satisfies the following property. For any finite collection $\{(\theta_j, \phi_j)\}_{j = 1}^m$ in $A$ and any permutation $\sigma$ of the set $\{1, 2, \ldots, m\}$, we have
\begin{equation} \label{eqn:cm}
\sum_{j = 1}^n c(\theta_j, \phi_j) \leq \sum_{j = 1}^m c(\theta_j, \phi_{\sigma(j)}).
\end{equation}
\end{definition}

It is well known that $c$-cyclical monotonicity is, under mild technical conditions, a necessary and sufficient solution criteria of the general optimal transport problem (see \cite[Chapter 1]{AG13}). In particular, a coupling $R$ of $(P, Q)$ is optimal if and only if the support of $R$ is $c$-cyclical monotone.

\subsection{Functionally generated portfolio}
At this point it is convenient to introduce the concept of functionally generated portfolio. Although it is possible to present the theory without reference to finance-motivated concepts, we stress that the portfolio map gives an additional structure to the transport problem not found in other cases. Also, the main examples of the theory as well as the key quantities (such as the induced Riemannian metric) are best expressed in terms of portfolios. Mathematically, the portfolio can be regarded as a normalized gradient of $\varphi$. In Section \ref{sec:application.fgp} we apply our information geometry to functionally generated portfolios.

Functionally generated portfolio was introduced in \cite{F99} and the following refined definition is taken from \cite{PW14}.

\begin{definition}[Functionally generated portfolio] \label{def:fgp}
By a portfolio map we mean a function $\ppi: \Delta_n \rightarrow \overline{\Delta}_n$. Let $\varphi: \Delta_n \rightarrow {\mathbb{R}}$ be exponentially concave. We say that a portfolio map $\ppi: \Delta_n \rightarrow \overline{\Delta}_n$ is generated by $\varphi$ if for any $p, q \in \Delta_n$ we have
\begin{equation} \label{eqn:fgp.def}
\sum_{i = 1}^n \ppi_i(p) \frac{q_i}{p_i} \geq e^{\varphi(q) - \varphi(p)}.
\end{equation}
We call $\varphi$ the log generating function of $\ppi$ and $\Phi := e^{\varphi}$ (which is positive and concave) the generating function. It is known that $\varphi$ is unique (for a given $\ppi$) up to an additive constant. If $\varphi$ is differentiable, then $\ppi$ is necessarily given by \eqref{eqn:fgp.weights}.
\end{definition}

Throughout this paper we impose the following regularity conditions on the exponentially concave function $\varphi$.

\begin{assumption} [Regularity conditions] \label{ass:regularity} {\ }
\begin{enumerate}
\item[(i)] The function $\varphi$ is smooth (i.e., infinitely differentiable) on $\Delta_n$. 
\item[(ii)] The (Euclidean) Hessian of $\Phi = e^{\varphi}$ is strictly negative definite everywhere on $\Delta_n$. In particular, $\Phi$ is strictly concave. Moreover, it can be shown that the function $\ppi$ defined by \eqref{eqn:fgp.weights} maps $\Delta_n$ into $\Delta_n$.
\end{enumerate}
\end{assumption}

Let us discuss these conditions briefly. Differentiability is needed to define differential geometric structures on $\Delta_n$ in terms of the derivatives of the L-divergence. Our theory requires the L-divergence $T\left(\cdot \mid \cdot\right)$ to be three times continuously differentiable, and for convenience we simply assume that $\varphi$ is smooth. Strict concavity guarantees that the L-divergence is non-degenerate, i.e., $T\left(q \mid p\right) = 0$ only if $q = p$, and strict positive definiteness of the Hessian implies that the induced Riemannian metric is non-degenerate.

Henceforth we let $\varphi: \Delta_n \rightarrow {\mathbb{R}}$ be an exponentially concave function satisfying Assumption \ref{ass:regularity} and let $\ppi: \Delta_n \rightarrow \Delta_n$ given by \eqref{eqn:fgp.weights} be the portfolio map generated by $\varphi$. The cost function $c$ always refers to the one defined in \eqref{eqn:psi.intro}, and a general cost function is denoted by $C$. Using \eqref{eqn:fgp.weights}, it can be shown that the L-divergence \eqref{eqn:L.divergence} of $\varphi$ can be expressed in the form
\begin{equation} \label{eqn:L.divergence.fgp}
T\left(q \mid p\right) = \log\left(\sum_{i = 1}^n \ppi_i(p) \frac{q_i}{p_i}\right) - \left(\varphi(q) - \varphi(p)\right).
\end{equation}

Now we give several examples of functionally generated portfolios and their log generating functions. Many more examples can be found in \cite[Chapter 3]{F02}. In particular, the constant-weighted portfolios play a special role and will be taken as the basic example of the theory.

\begin{example} [Examples of functionally generated portfolios] \label{ex:fgp.example} {\ }
\begin{enumerate}
\item[(i)] (Market portfolio) The identity map $\ppi(p)\equiv p$ is generated by the constant function $\varphi(p) \equiv 0$. (Here Assumption \ref{ass:regularity} does not hold.)
\item[(ii)] (Constant-weighted portfolio) The constant map $\ppi(p) \equiv \pi \in \Delta_n$ is generated by the cross-entropy $\varphi(p) = \sum_{i = 1}^n \pi_i \log p_i$. The special case $\pi = \left(\frac{1}{n}, \ldots, \frac{1}{n}\right)$ is called the equal-weighted portfolio. 
\item[(iii)] (Diversity-weighted portfolio) Let $\lambda \in (0, 1)$ be a fixed parameter. Consider the portfolio map $\ppi: \Delta_n \rightarrow \Delta_n$ defined by
\[
\ppi_i(p) = \frac{p_i^{\lambda}}{\sum_{j = 1}^n p_j^{\lambda}}, \quad i = 1, \ldots, n.
\]
It can be shown that the generating function is $\varphi(p) = \frac{1}{\lambda} \log \left(\sum_{j = 1}^n p_j^{\lambda}\right)$. 
\item[(iv)] (Convex combinations) It is known that the set of functionally generated portfolios is convex. Indeed, let $\ppi^{(1)}$ be generated by $\varphi^{(1)}$ and $\ppi^{(2)}$ be generated by $\varphi^{(2)}$. Then for $\lambda \in (0, 1)$, the portfolio map $\ppi = (1 - \lambda) \ppi^{(1)} + \lambda \ppi^{(2)}$ is generated by $\varphi = (1 - \lambda) \varphi^{(1)} + \lambda \varphi^{(2)}$.
Its generating function $\Phi = e^{\varphi}$ is then the geometric mean $\left(\Phi^{(1)}\right)^{1 - \lambda} \left(\Phi^{(2)}\right)^{\lambda}$. This fact was used in \cite{W14, W15} to formulate and study nonparametric estimation of functionally generated portfolio.
\end{enumerate}
\end{example}

The following result is taken from \cite{PW14}.

\begin{proposition}\label{prop:fgp.mcm}
For any portfolio map $\ppi: \Delta_n \rightarrow \Delta_n$ the following statements are equivalent.
\begin{enumerate}
\item[(i)] There exists an exponentially concave function $\varphi$ on $\Delta_n$ which generates $\ppi$ in the sense of \eqref{eqn:fgp.def}.
\item[(ii)] The portfolio map is multiplicatively cyclical monotone (MCM) in the following sense: for any sequence $\{\mu(t)\}_{t = 0}^{m + 1}$ in $\Delta_n$ satisfying $\mu(m + 1) = \mu(0)$, we have
\begin{equation*}
\prod_{t = 0}^m \left(\sum_{i = 1}^n \ppi_i(\mu(t)) \frac{\mu_i(t + 1)}{\mu_i(t)}\right) \geq 1.
\end{equation*}
\item[(iii)] Define a map $\theta \mapsto \phi$ by
\begin{equation} \label{eqn:portfolio.transport}
\phi_i = \theta_i - \log \frac{\ppi_i(\theta)}{\ppi_n(\theta)}, \quad i = 1, \ldots, n - 1.
\end{equation}
Here $\ppi$ is regarded as a function of the exponential coordinate. In words, we define $\phi$ in such a way that the exponential coordinate of $\ppi(\theta)$ is $\theta - \phi$. Then the graph of this map is $c$-cyclical monotone.
\end{enumerate}
\end{proposition}

Using this result, we showed in \cite{PW14} how the optimal transport problem \eqref{eqn:transport.problem} can be solved in terms of functionally generated portfolios. Here is a simple but interesting explicit example which is a direct generalization of the one-dimensional case treated in \cite[Section 4]{PW14}.

\begin{example}[Product of Gaussian distributions] \label{ex:gaussian}
In the transport problem \eqref{eqn:transport.problem}, let $P$ be a product of one-dimensional Gaussian distributions:
\[
P = \bigotimes_{i = 1}^{n-1} N(a_i, \sigma_i^2),
\]
where $a_i \in {\mathbb{R}}$ and $\sigma_i > 0$. Also let
\[
Q = \bigotimes_{i = 1}^{n-1} N(b_i, (1 -  \lambda) \sigma_i^2),
\]
where $b_i \in {\mathbb{R}}$ and $0 < \lambda < 1$. Then the optimal transport map for the measures $P$ and $Q$ is given by the map \eqref{eqn:portfolio.transport}, where the portfolio map $\ppi$ is the following variant of the diversity-weighted portfolio discussed in Example \ref{ex:fgp.example}(iii):
\begin{equation} \label{eqn:general.diversity.weighted}
\ppi_i(p) = \frac{w_i p_i^{\lambda}}{\sum_{j = 1}^n w_j p_j^{\lambda}}, \quad \varphi(p) = \frac{1}{\lambda} \log\left(\sum_{j = 1}^n w_j p_j^{\lambda}\right),
\end{equation}
where the coefficients $w_i$ are chosen such that $(1 - \lambda) a_i - \log \frac{w_i}{w_n} = b_i$ for all $i$. 
\end{example}

\section{Optimal transport and duality} \label{sec:duality}
\subsection{$c$-concavity and duality}
Now we make use of the notion of $c$-concavity in optimal transport theory. The definitions we use are standard and can be found in \cite[Chapter 1]{AG13}. Again, $c$ refers to our cost function \eqref{eqn:psi.intro}. Also recall that ${\mathcal{X}} = {\mathbb{R}}^{n-1}$ and ${\mathcal{Y}} = {\mathbb{R}}^{n-1}$ are the underlying spaces of the variables $\theta$ and $\phi$ respectively. For $f: {\mathcal{X}} \rightarrow {\mathbb{R}} \cup \{\pm \infty\}$ we define its $c$-transform by
\[
f^*(\phi) := \inf_{\theta \in {\mathcal{X}}} \left(c(\theta, \phi) - f(\theta)\right), \quad \phi \in {\mathcal{Y}}.
\]
Similarly, the $c$-transform of a function $g: {\mathcal{Y}} \rightarrow {\mathbb{R}} \cup \{\pm \infty\}$ is defined by
\[
g^*(\theta) := \inf_{\phi \in {\mathcal{Y}}} \left(c(\theta, \phi) - g(\phi)\right), \quad \theta \in {\mathcal{X}}.
\]
We say that $f: {\mathcal{X}} \rightarrow {\mathbb{R}} \cup \{-\infty\}$ is $c$-concave if there exists $g: {\mathcal{Y}} \rightarrow {\mathbb{R}} \cup \{-\infty\}$ such that $f = g^*$ (similar for $c$-concave functions on ${\mathcal{Y}}$). A function $h$ (on ${\mathcal{X}}$ or ${\mathcal{Y}}$) is $c$-concave if and only if $h^{**} = h$.

If $f: {\mathcal{X}} \rightarrow {\mathbb{R}} \cup \{-\infty\}$ is $c$-concave, its $c$-superdifferential is defined by
\begin{equation} \label{eqn:c.super.differential}
\partial^{c} f := \left\{(\theta, \phi) \in {\mathcal{X}} \times {\mathcal{Y}}: f(\theta) + f^*(\phi) = c(\theta, \phi) \right\}.
\end{equation}
For $\theta \in {\mathcal{X}}$ we define $\partial^{c} f(\theta) := \left\{\phi \in {\mathcal{Y}}: (\theta, \phi) \in \partial^c f \right\}$. If this set is a singleton $\{\phi\}$, we call $\phi$ the $c$-supergradient of $f$ at $\theta$ and  write $\phi = \nabla^c f(\theta)$. Similar definitions hold for a $c$-concave function $g$ on ${\mathcal{Y}}$.

Let $f: {\mathcal{X}} \rightarrow {\mathbb{R}} \cup \{-\infty\}$ be $c$-concave. By definition of $f^*$, we have
\begin{equation} \label{eqn:fenchel.identity}
f(\theta) + f^*(\phi) \leq c(\theta, \phi)
\end{equation}
for every pair $(\theta, \phi) \in {\mathcal{X}} \times {\mathcal{Y}}$, and equality holds in \eqref{eqn:fenchel.identity} if and only if $(\theta, \phi) \in \partial^c f$. This is a generalized version of Fenchel's identity (see \cite[Section 12]{R70}) and will be used frequently in this paper. 

Our first lemma relates exponential concavity on $\Delta_n$ with $c$-concavity on ${\mathcal{X}}$ and on ${\mathcal{Y}}$. Note that the cost function is asymmetric, and $c$-concavity on ${\mathcal{X}}$ is equivalent to $c$-concavity on ${\mathcal{Y}}$ after a change of variable.

\begin{lemma} [Exponential concavity and $c$-concavity] \label{lem:c.concave}
For $\varphi: \Delta_n \rightarrow {\mathbb{R}} \cup \{-\infty\}$ the following statements are equivalent.
\begin{enumerate}
\item[(i)] $\varphi$ is exponentially concave on $\Delta_n$.
\item[(ii)] The function $f: {\mathcal{X}} \rightarrow {\mathbb{R}} \cup \{-\infty\}$ defined by
\[
f(\theta) = \varphi({\bf p}(\theta)) + \psi(\theta)
\]
is $c$-concave on ${\mathcal{X}}$.
\item[(iii)] The function $g: {\mathcal{Y}} \rightarrow {\mathbb{R}} \cup \{-\infty\}$ defined by
\[
g(\phi) = \varphi({\bf p}(-\phi)) + \psi(-\phi),
\]
where $-\phi$ is the exponential coordinate, is $c$-concave on ${\mathcal{Y}}$.
\end{enumerate}
\end{lemma}
\begin{proof}
We prove the implication (i) $\Rightarrow$ (ii) and the others can be proved similarly. Suppose (i) holds and consider the non-negative concave function $\Phi = e^{\varphi}$ on $\Delta_n$. By \cite[Theorem 10.3]{R70}, we can extend $\Phi$ continuously up to $\overline{\Delta}_n$, the closure of $\Delta_n$ in ${\mathbb{R}}^n$. We further extend $\Phi$ to the affine hull $H$ of $\Delta_n$ in ${\mathbb{R}}^n$ by setting $\Phi(p) = -\infty$ for $p \notin \overline{\Delta}_n$. The extended function $\Phi$ is then a closed concave function on $H$. By convex duality (see \cite[Theorem 12.1]{R70}), there exists a family ${\mathcal{C}}$ of affine functions on $H$ such that
\begin{equation} \label{eqn:convex.duality}
\Phi(p) = \inf_{\ell \in {\mathcal{C}}} \ell(p), \quad p \in \Delta_n.
\end{equation}
Since $\Phi$ is non-negative on $\Delta_n$, each $\ell \in {\mathcal{C}}$ is non-negative on $\Delta_n$. Replacing $\ell$ by the sequence $\ell_k = \ell + \frac{1}{k}$, $k \geq 1$, we may assume without loss of generality that each $\ell \in {\mathcal{C}}$ is strictly positive on $\Delta_n$. We parameterize each $\ell \in {\mathcal{C}}$ in the form $\ell(p) = \sum_{i = 1}^n a_i p_i$ where $a_1, \ldots, a_n$ are positive constants. (Note that an extra constant term is not required since $p_1 + \cdots + p_n = 1$.) Writing $\phi_i := - \log \frac{a_i}{a_n}$, $i = 1, \ldots, n - 1$ and switching to exponential coordinates, we have
\begin{equation*}
\begin{split}
\log \ell(p) &= \log\left(\sum_{i = 1}^n a_i p_i\right) \\
  &= \log\left(1 + \sum_{i = 1}^{n-1} \frac{a_i}{a_n} \frac{p_i}{p_n} \right) + \log p_n  + \log a_n \\
 &= \log\left(1 + \sum_{i = 1}^{n-1} e^{\theta_i - \phi_i} \right) - \psi(\theta) + \log a_n \\
  &= c(\theta - \phi) - \psi(\theta) + \log a_n,
\end{split}
\end{equation*}
It follows from \eqref{eqn:convex.duality} that
\begin{equation} \label{eqn:f.representation}
f(\theta) = \varphi(\theta) + \psi(\theta) = \inf_{\ell \in {\mathcal{C}}} \left( c(\theta - \phi) + \log a_n\right).
\end{equation}
Define $h: {\mathcal{Y}} \rightarrow {\mathbb{R}} \cup \{-\infty\}$ by setting
\[
h(\phi) = \inf\left\{-\log a_n: \exists \ \ell(p) = \sum_{i = 1}^n a_i p_i \in {\mathcal{C}} \text{ s.t. } \phi_i = - \log \frac{a_i}{a_n} \ \forall \  i \right\},
\]
where the infimum of the empty set is $-\infty$. From \eqref{eqn:f.representation}, we have
\[
f(\theta) = \inf_{\phi \in {\mathcal{Y}}} \left( c(\theta - \phi) - h(\phi)\right) = h^*(\theta)
\]
which shows that $f$ is $c$-concave on ${\mathcal{X}}$.
\end{proof}

The following is the $c$-concave analogue of the classical Legendre transformation \cite{R70}. Its proof is standard but lengthy and will be given in the Appendix.

\begin{theorem}[$c$-Legendre transformation] \label{thm:duality}
Let $\varphi$ be an exponentially concave function $\varphi$ satisfying Assumption \ref{ass:regularity}, and let $\ppi$, defined by \eqref{eqn:fgp.weights}, be the portfolio map generated by $\varphi$. Given $\varphi$, consider the $c$-concave function
\begin{equation} \label{eqn:f.varphi}
f(\theta) := \varphi(\theta) + \psi(\theta)
\end{equation}
defined on ${\mathcal{X}} = {\mathbb{R}}^{n-1}$ via the exponential coordinate system.
\begin{enumerate}
\item[(i)] The $c$-supergradient of $f$ is given by \eqref{eqn:portfolio.transport}, i.e.,
\begin{equation} \label{eqn:f.varphi.gradient}
\nabla^c f(\theta) = \left(\theta_i - \log \frac{\ppi_i(\theta)}{\ppi_n(\theta)}\right)_{1 \leq i \leq n - 1}, \quad \theta \in {\mathcal{X}}.
\end{equation}
Moreover, the map $\nabla^c f: {\mathcal{X}} \rightarrow {\mathcal{Y}}$ is injective.
\item[(ii)] Let ${\mathcal{Y}}' \subset {\mathcal{Y}}$ be the range of $\nabla^c f$. Then the $c$-supergradient of $f^*$ is given on ${\mathcal{Y}}'$ by
\begin{equation*}
\nabla^c f^*(\phi) = \left(\nabla^c f\right)^{-1}(\phi), \quad \phi \in {\mathcal{Y}}'.
\end{equation*}
\end{enumerate}
In fact, the map $\nabla^c f$ is a diffeomorphism from ${\mathcal{X}}$ to ${\mathcal{Y}}'$ whose inverse is $\nabla^c f^*$. Also, the function $f^*$ is smooth on the open set ${\mathcal{Y}}'$.
\end{theorem}

Although ${\mathcal{Y}}'$ is in general a strict subset of ${\mathcal{Y}}$, by Theorem \ref{thm:duality} the dual variable $\phi = \nabla^c f(\theta)$ defines a global coordinate system of the manifold $\Delta_n$. In Theorem \ref{thm:geodesics} we will use another coordinate system on $\Delta_n$ called the dual Euclidean coordinate system. Thus we have four coordinate systems on $\Delta_n$: Euclidean, primal, dual and dual Euclidean (see Definition \ref{def:coordinate}). In the following we will frequently switch between coordinate systems to facilitate computations. To avoid confusions let us state once for all the conventions used. We let $\varphi$ and $f = \varphi + \psi$ be given.

\begin{figure}[t!]
\begin{tikzpicture}[scale = 0.6]
\draw[<-] (3.7, -1) to (3.7, -5);
\draw[->] (4.1, -1) to (4.1, -5);
\node[left] at (3.7, -2.5) {{\footnotesize ${\bf p} \circ -\mathrm{Id}$}};
\node[right] at (4.1, -2.5) {{\footnotesize $-\boldsymbol{\theta}$}};

\draw[->] (-2.2, -1) to (2.2, -5);
\node[above] at (0, -3) {{\footnotesize $\boldsymbol{\phi}$}};

\draw[->] (-4.3, -1) to (-4.3, -5);
\draw[<-] (-3.9, -1) to (-3.9, -5);
\node[left] at (-4.3, -2.5) {{\footnotesize $\boldsymbol{\theta}$}};
\node[right] at (-3.9, -2.5) {{\footnotesize ${\bf p}$}};

\draw[->] (-2.2, -0.5) to (2.2, -0.5);
\node[above] at (0, -0.5) {{\footnotesize ${\bf p}^*$}};

\node [above] at (-4, -1) {$\Delta_n$};
\node [above] at (-4, 0) {{\footnotesize primal Euclidean}};
\node [above] at (4, -1) {$\Delta_n$};
\node [above] at (4, 0) {{\footnotesize dual Euclidean}};

\node [below] at (-4, -5) {${\mathcal{X}} = {\mathbb{R}}^{n-1}$};
\node [below] at (-4, -6) {{\footnotesize primal exponential}};

\node [below] at (4.2, -5) {${\mathcal{Y}}' \subset {\mathbb{R}}^{n-1}$};
\node [below] at (4, -6) {{\footnotesize dual exponential}};

\draw[->] (-2.2, -5.3) to (2.2, -5.3);
\draw[->] (2.2, -5.7) to (-2.2, -5.7);

\node[above] at (0, -5.3) {{\footnotesize $\nabla^c f$}};
\node[below] at (0, -5.7) {{\footnotesize $\nabla^c f^*$}};

\end{tikzpicture}
\caption{Coordinate systems on $\Delta_n$.} \label{fig:coordinate}
\end{figure}
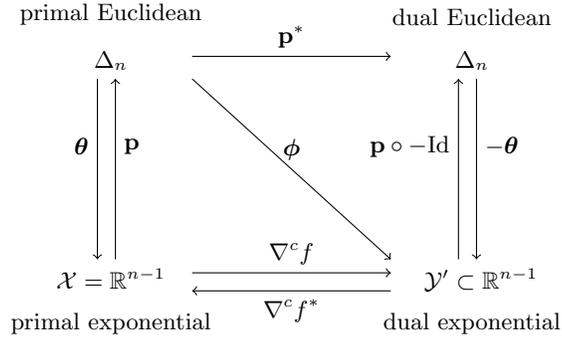

\begin{definition}[Coordinate systems] \label{def:coordinate}
For the unit simplex $\Delta_n$ (defined by \eqref{eqn:simplex}) we call the identity map
\[
p = (p_1, \ldots, p_n), \quad p_i > 0, \quad \sum_{i = 1}^n p_i = 1
\]
the (primal) Euclidean coordinate system with range $\Delta_n$. We let
\[
\theta = \boldsymbol{\theta}(p) = \left(\log \frac{p_1}{p_n}, \ldots, \log \frac{p_{n-1}}{p_n}\right)
\]
be the primal (exponential) coordinate system with range ${\mathcal{X}}$ and 
\[
\phi = \boldsymbol{\phi}(p) := \nabla^c f(\theta)
\]
be the dual (exponential) coordinate system with range ${\mathcal{Y}}'$. The dual Euclidean coordinate system is defined by the composition
\[
p^* = {\bf p}^*(p) := {\bf p}(-\boldsymbol{\phi}(p)).
\]
See Figure \ref{fig:coordinate} for an illustration. From now on $p$, $p^*$, $\theta$ and $\phi$ always represent the same point of $\Delta_n$. In particular, unless otherwise specified $\theta$ and $\phi$ are dual to each other in the sense that $\phi = \nabla^c f(\theta)$. By convention we let $\theta_n = \phi_n = 0$ for any $p \in \Delta_n$.
\end{definition}

\begin{notation} [Switching coordinate systems] \label{not:coordinate}
We identify the spaces $\Delta_n$, ${\mathcal{X}}$ and ${\mathcal{Y}}'$ using the coordinate systems in Definition \ref{def:coordinate}. If $h$ is a function on any one of these spaces, we write $h(p) = h(\theta) = h(\phi) = h(p^*)$ depending on the coordinate system used.
\end{notation}

We also record a useful fact. A formula analogous to the first statement is derived in \cite{S14}.

\begin{lemma} \label{lem:weight.as.derivative}
For $1 \leq i \leq n - 1$, we have
\begin{equation*}
\begin{split}
\frac{\partial}{\partial \theta_i} f(\theta) &= \ppi_i(\theta), \quad \theta \in {\mathcal{X}}, \\
\frac{\partial}{\partial \phi_i} f^*(\phi) &= -\ppi_i(\phi), \quad \phi \in {\mathcal{Y}}'.
\end{split}
\end{equation*}
\end{lemma}
\begin{proof}
The first statement is derived in the proof of Theorem \ref{thm:duality}. The second statement can be proved by differentiating $f^*(\phi) = c(\theta, \phi) - f(\theta)$ (Fenchel's identity).
\end{proof}

\subsection{$c$-divergence} \label{sec:c.divergence}
By duality, we show that a pair of natural divergences on $\Delta_n$ can be defined for the $c$-concave functions $f$ and $f^*$. Moreover, they coincide with L-divergence. Clearly we can consider other cost functions $C$ other than $c$. When $C$ is the squared Euclidean distance, the analogue of Definition \ref{def:c.divergence} below gives the classical Bregman divergence. This covers both L-divergence and Bregman divergence under the same framework. To the best of our knowledge these definitions, which depend crucially on the interplay between transport and divergence, are new. We will use the triple representation $(p, \theta, \phi)$ for each point in $\Delta_n$.

\begin{definition} [$c$-divergence] \label{def:c.divergence}
Consider the $c$-concave function $f$ defined by \eqref{eqn:f.varphi} and its $c$-transform $f^*$. 
\begin{enumerate}
\item[(i)] The $c$-divergence of $f$ is defined by
\begin{equation} \label{eqn:c.divergence}
D\left( p \mid p'\right) = c(\theta, \phi') - c(\theta', \phi') - (f(\theta) - f(\theta')), \quad p, p' \in \Delta_n.
\end{equation}
\item[(ii)] The $c$-divergence of $f^*$ is defined by
\begin{equation} \label{eqn:c.dual.divergence}
D^*\left( p \mid p'\right) = c(\theta', \phi) - c(\theta', \phi') - (f^*(\phi) - f^*(\phi')), \quad p, p' \in \Delta_n.
\end{equation}
\end{enumerate}
\end{definition}

From Fenchel's identity \eqref{eqn:fenchel.identity} we see that $D$ and $D^*$ are non-negative and non-degnereate, i.e., they vanish only on the diagonal of $\Delta_n \times \Delta_n$. The following is a generalization of the self-dual expression of Bregman divergence (see \cite[Theorem 1.1]{A16}).

\begin{proposition} [Self-dual expressions] \label{prop:self.dual}
We have
\begin{eqnarray} 
D\left(p \mid p'\right) &=& c(\theta, \phi') - f(\theta) - f^*(\phi'), \label{eqn:c.divergence.2} \\  
D^*\left(p \mid p'\right) &=& c(\theta', \phi) - f^*(\phi) - f(\theta'). \label{eqn:c.dual.divergence.2}
\end{eqnarray}
In particular, for $p, p' \in \Delta_n$ we have $D\left( p \mid  p'\right) = D^*\left(p' \mid p\right)$.
\end{proposition}
\begin{proof}
To prove \eqref{eqn:c.divergence.2}, we use the Fenchel identity $f(\theta') + f^*(\phi') = c(\theta', \phi')$. Starting from \eqref{eqn:c.divergence}, we have
\begin{equation*}
\begin{split}
D\left( p \mid p'\right) &= c(\theta, \phi') - c(\theta', \phi') - (f(\theta) - f(\theta')) \\
  &= c(\theta, \phi') - f(\theta) - f^*(\phi').
\end{split}
\end{equation*}
The proof of \eqref{eqn:c.dual.divergence.2} is similar.
\end{proof}

Now we show that L-divergence is a $c$-divergence where $c(\theta, \phi) = \psi(\theta - \phi)$.

\begin{theorem} [L-divergence as $c$-divergence] \label{prop:c.divergence}
The $c$-divergence of $f$ is the L-divergence of $\varphi$. Namely, for $p, p' \in \Delta_n$ we have
\[
D\left( p \mid p'\right) = T\left(p \mid p'\right).
\]
\end{theorem}
\begin{proof}
Using the primal-dual relation \eqref{eqn:f.varphi.gradient}, we have
\[
\psi(\theta - \phi') = \log \left(\sum_{i = 1}^n e^{\theta_i-\theta_i' + \log \frac{\ppi_i(\theta')}{\ppi_n(\theta')}}\right) = \log \left(\ppi(p') \cdot \frac{p}{p'}\right) - \log \left(\ppi_n(p') \frac{p_n}{p_n'} \right).
\]
Next, by Fenchel's identity (see \eqref{eqn:fenchel.identity}), we have
\[
f^*(\phi') = \psi(\theta' - \phi') - f(\theta') = \psi(\theta' - \phi') - \varphi(\theta') - \psi(\theta').
\]
Using these identities and \eqref{eqn:L.divergence.fgp}, we compute
\begin{equation*}
\begin{split}
D\left(p \mid p'\right) &= \psi(\theta - \phi') - f(\theta) - f^*(\phi')  \\
&= \log \left(\ppi(p') \cdot \frac{p}{p'}\right) - \log \left(\ppi_n(p') \frac{p_n}{p_n'} \right) \\
&\quad \quad  - \left(\varphi(\theta) + \psi(\theta)\right) - \left( \psi(\theta' - \phi') - \varphi(\theta') - \psi(\theta') \right) \\
&= \log \left(\ppi(p') \cdot \frac{p}{p'}\right) - \left(\varphi(\theta) - \varphi(\theta')\right) \\
&= T\left(p \mid p'\right).  
\end{split} 
\end{equation*}
\end{proof}

For computations it is convenient to express $T\left(p \mid p'\right)$ solely in terms of either the primal or dual coordinates. We omit the details of the computations.

\begin{lemma} [Coordinate representations] \label{lem:coord.rep}
For $p, p' \in \Delta_n$ we have
\begin{equation*} 
\begin{split}
T\left(p \mid p'\right) &= \log\left(\sum_{\ell = 1}^n \ppi_{\ell}(\theta') e^{\theta_{\ell} - \theta_{\ell}'}\right) - \left(f(\theta) - f(\theta')\right),\\
T\left(p \mid p'\right) &= \log\left(\sum_{\ell = 1}^n \ppi_{\ell}(\phi) e^{\phi_{\ell} - \phi_{\ell}'}\right) - \left(f^*(\phi') - f^*(\phi)\right).
\end{split}
\end{equation*}
\end{lemma}

\subsection{Transport interpretation of the generalized Pythagorean theorem} \label{sec:pqr.transport}
Using Proposition \ref{prop:c.divergence} we give an interesting transport interpretation of the expression \eqref{eqn:pyth} in the generalized Pythagorean theorem (Theorem \ref{thm:pyth}). Let $p, q, r \in \Delta_n$ be given. Let $(\theta^{(j)}, \phi^{(j)})_{1 \leq j \leq 3}$ be the primal and dual coordinates of $p$, $q$ and $r$ respectively. By Proposition \ref{prop:fgp.mcm}, the coupling $(\theta, \phi = \nabla f^c(\theta))$ is $c$-cyclical monotone. Hence coupling $\theta^{(j)}$ with $\phi^{(j)}$ is optimal.

Consider two (suboptimal) perturbations of the optimal coupling:
\begin{itemize}
\item[(i)] (Cyclical perturbation) Couple $\theta^{(1)}$ with $\phi^{(3)}$, $\theta^{(2)}$ with $\phi^{(1)}$, and $\theta^{(3)}$ with $\phi^{(2)}$. The associated cost is
\[
c(\theta^{(1)}, \phi^{(3)}) + c(\theta^{(2)}, \phi^{(1)}) + c(\theta^{(3)}, \phi^{(2)}).
\]
\item[(ii)] (Transposition) Couple $\theta^{(1)}$ with $\phi^{(3)}$, $\theta^{(3)}$ with $\phi^{(1)}$, and keep the coupling $(\theta^{(2)}, \phi^{(2)})$. The associated cost is
\[
c(\theta^{(1)}, \phi^{(3)}) + c(\theta^{(3)}, \phi^{(1)}) + c(\theta^{(2)}, \phi^{(2)}).
\]
\end{itemize} 
Now we ask which perturbation has lower cost. The difference (i) $-$ (ii) is
\[
c(\theta^{(2)}, \phi^{(1)}) + c(\theta^{(3)}, \phi^{(2)}) - c(\theta^{(3)}, \phi^{(1)}) - c(\theta^{(2)}, \phi^{(2)}).
\]
By Proposition \ref{prop:c.divergence}, this is nothing but the difference $T\left(q \mid p\right) + T\left(r \mid q\right) - T\left(r \mid p\right)$. Thus the generalized Pythagorean theorem gives an information geometric characterization of the relative costs of the two perturbations.

\subsection{Examples} \label{sec:examples}
We consider the portfolios in Example \ref{ex:fgp.example}.

\begin{example}[Constant-weighted portfolio]
Let $\pi \in \Delta_n$ be a constant-weighted portfolio. Then $\varphi$ is the cross entropy and we have 
\[
f(\theta) = \varphi(\theta) + \psi(\theta) = \sum_{i = 1}^{n-1} \pi_i \theta_i,
\]
which is an affine function on ${\mathcal{X}}$. Its $c$-transform is also affine. Indeed, we have
\[
f^*(\phi) = \sum_{i = 1}^{n - 1} \pi_i(-\phi_i) + H(\pi),
\]
where $H(\pi) := -\sum_{i = 1}^n \pi_i \log \pi_i$ is the Shannon entropy of $\pi$. For this reason we say that the constant-weighted portfolios are self-dual. The transport map in this case is given by a translation: $\phi = \theta - \left(\log \frac{\pi_1}{\pi_n}, \ldots, \log \frac{\pi_{n-1}}{\pi_n}\right)$. Its L-divergence \eqref{eqn:excess.growth} is given in primal coordinates (see Lemma \ref{lem:coord.rep}) by
\[
T\left(p \mid p'\right) = \log\left(\sum_{\ell = 1}^n \pi_{\ell} e^{\theta_{\ell} - \theta_{\ell}'}\right) - \sum_{\ell = 1}^n \pi_{\ell} (\theta_{\ell} - \theta_{\ell}'),
\]
which is translation invariant. This property is equivalent to the following num\'{e}raire invariance property \cite[Lemma 3.2]{PW13}: for any $w_1, \ldots, w_n > 0$, we have $T\left(q \mid p\right) = T\left(\widetilde{q} \mid \widetilde{p}\right)$ under the mapping
\[
p \mapsto \widetilde{p} = \left(\frac{w_ip_i}{w_1p_1 + \cdots + w_np_n}\right)_{1 \leq i \leq n}.
\]
In fact, it is not difficult to show that this property characterizes the constant-weighted portfolios among L-divergences of exponentially concave functions. Also see \cite[Proposition 4.6]{PW13} for a chain rule analogous to that of relative entropy.
\end{example}

\begin{example}[Diversity-weighted portfolio]
We have $f(\theta) = \frac{1}{\lambda} \psi(\lambda \theta)$. Since 
\[
\log \frac{\ppi_i(\theta)}{\ppi_n(\theta)} = \lambda \theta_i,
\]
the map $\theta \mapsto \phi$ is a scaling: $\phi = (1 - \lambda) \theta$. For the generalized diversity-weighted portfolio in Example \ref{ex:gaussian} the transport map $\theta \mapsto \phi$ is the composition of a scaling and a translation.
\end{example}

\section{Geometric structure induced by L-divergence} \label{sec:geometry.computation}
In this section we derive the geometric structure induced by a given L-divergence $T\left(\cdot \mid \cdot\right)$. As always we impose the regularity conditions in Assumption \ref{ass:regularity}. Using the primal and dual coordinate systems (Definition \ref{def:coordinate}), we compute explicitly the Riemannian metric $g$, the primal connection $\nabla$ (not to be confused with the Euclidean gradient) and the dual connection $\nabla^*$. We call $(g, \nabla, \nabla^*)$ the induced geometric structure. An important fact in information geometry is that the Levi-Civita connection $\nabla^{(0)}$ is not necessarily the right one to use. Nevertheless, by duality we always have $\nabla^{(0)} = \frac{1}{2}\left(\nabla + \nabla^*\right)$.

\subsection{Preliminaries}
For differential geometric concepts such as Riemannian metric and affine connection we refer the reader to \cite[Chapters 5]{A16} whose notations are consistent with ours. For computational convenience we define the geometric structure in terms of coordinate representations. The geometric structure is determined by the L-divergence $T\left(\cdot \mid \cdot\right)$ and is independent of the choice of coordinates; for intrinsic formulations we refer the reader to \cite[Chapter 11]{CU14}. The following definition (which makes sense for a general divergence on a manifold) is taken from \cite[Section 6.2]{A16}.

\begin{definition}[Induced geometric structure]
Given a coordinate system $\xi = (\xi_1, \ldots, \xi_{n-1})$ of $\Delta_n$, the coefficients of the geometric structure $(g, \nabla, \nabla^*)$ are given as follows.
\begin{enumerate}
\item[(i)] The Riemannian metric is given by
\begin{equation} \label{eqn:metric}
g_{ij}(\xi) = - \left.\frac{\partial}{\partial \xi_i} \frac{\partial}{\partial \xi_j'} T\left(\xi \mid \xi'\right) \right|_{\xi = \xi'}, \quad i, j = 1, \ldots, n - 1.
\end{equation}
By Assumption \ref{ass:regularity} the matrix $\left(g_{ij}(\xi)\right)$ is strictly positive definite. The Riemannian inner product and length are denoted by $\langle \cdot, \cdot \rangle$ and $\|\cdot\|$ respectively.
\item[(ii)] The primal connection is given by
\begin{equation} \label{eqn:primal.connection}
\Gamma_{ijk}(\xi) =  - \left.\frac{\partial}{\partial \xi_i} \frac{\partial}{\partial \xi_j} \frac{\partial}{\partial \xi_k'} T\left(\xi \mid \xi'\right) \right|_{\xi = \xi'}, \quad i, j, k = 1, \ldots, n - 1.
\end{equation}
\item[(iii)] The dual connection is given by
\begin{equation} \label{eqn:dual.connection}
\Gamma_{ijk}^*(\xi) =  - \left.\frac{\partial}{\partial \xi_k} \frac{\partial}{\partial \xi_i'} \frac{\partial}{\partial \xi_j'} T\left(\xi \mid \xi'\right) \right|_{\xi = \xi'}, \quad i, j, k = 1, \ldots, n - 1.
\end{equation}
\end{enumerate}
\end{definition}

For a general divergence the above definitions were first introduced in \cite{E83, E92}. If we define the dual divergence by $T^*\left(p \mid p'\right) := T\left(p' \mid p\right)$, the dual connection of $T$ is the primal connection of $T^*$. The primal and dual connections are dual to each other with respect to the Riemannian metric $g$ (see \cite[Theorem 6.2]{A16}). While any divergence induces a geometric structure, it may not enjoy nice properties. For the geometric structure induced by a Bregman divergence, it can be shown that the Riemann-Christoffel curvatures of the primal and dual connections both vanish. Thus we say that the induced geometry is dually flat \cite[Chapter 1]{A16}. We will show that L-divergence gives rise to a different geometry with many interesting properties.

\subsection{Notations} 
We begin by clarifying the notations. Following our convention (see Notation \ref{not:coordinate}), we write $T\left(p \mid p'\right) = T\left(\theta \mid \theta'\right) = T\left(\phi \mid \phi'\right)$ depending on the coordinate system used. The primal and dual coordinate representations have been computed in Lemma \ref{lem:coord.rep}. 

The Riemannian metric will be computed using both the primal and dual coordinate systems. To be explicit about the coordinate system we use $g_{ij}(\theta)$ to denote its coefficients in primal coordinates, and $g_{ij}^*(\phi)$ for its coefficients in dual coordinates:
\[
g_{ij}(\theta) := -\left. \frac{\partial^2}{\partial \theta_i \partial \theta_j'} T\left(\theta \mid \theta'\right) \right|_{\theta = \theta'}, \quad g_{ij}^*(\phi) := -\left. \frac{\partial^2}{\partial \phi_i \partial \phi_j'} T\left(\phi \mid \phi'\right) \right|_{\phi = \phi'}.
\]
The inverses of the matrices $\left(g_{ij}(\theta)\right)$ and $\left( g_{ij}^*(\phi)\right)$ are denoted by $\left(g^{ij}(\theta)\right)$ and $\left( g^{*ij}(\phi)\right)$ respectively.

The primal connection $\nabla$ will be computed using the primal coordinate system:
\[
\Gamma_{ijk}(\theta) :=  -\left.\frac{\partial^3}{\partial \theta_i \partial \theta_j \partial \theta_k'} T\left(\theta \mid \theta'\right)\right|_{\theta = \theta'}, \quad \Gamma_{ij}^k(\theta) := \sum_{m = 1}^{n-1} \Gamma_{ijm}(\theta)g^{mk}(\theta).
\]
The dual connection $\nabla^*$ will be computed using the dual coordinate system:
\[
\Gamma_{ijk}^*(\phi) := -\left.\frac{\partial^3}{\partial \phi_k \partial \phi_i' \partial \phi_j'} T\left(\phi \mid \phi'\right)\right|_{\phi = \phi'}, \quad \Gamma_{ij}^{*k}(\phi) := \sum_{m = 1}^{n-1} \Gamma_{ijm}^*(\phi)g^{*mk}(\phi).
\]

The following notations are useful. For $1 \leq i \leq n$ we define
\begin{equation} \label{eqn:Pi}
\Pi_i(\theta, \theta') := \frac{\ppi_i(\theta') e^{\theta_i - \theta_i'}}{\sum_{\ell = 1}^n \ppi_{\ell}(\theta') e^{\theta_{\ell} - \theta_{\ell}'}}, \quad \Pi^*_i(\phi, \phi') :=\frac{\ppi_i(\phi)e^{\phi_i - \phi_i'}}{\sum_{\ell = 1}^n \ppi_{\ell}(\phi)e^{\phi_{\ell} - \phi_{\ell}'}}.
\end{equation}
As always we adopt the convention $\theta_n = \theta_n' = \phi_n = \phi_n' = 0$. Note that $\Pi_i(\theta, \theta')$ involves the portfolio at $\theta'$ (the second variable) while $\Pi_i^*(\phi, \phi')$ involves the portfolio at $\phi$ (the first variable). The partial derivatives of $\Pi_i$ and $\Pi_i^*$ are given in the next lemma and can be verified by direct differentiation. We let $\delta_{ij}$ be the Kronecker delta and $\delta_{ijk} = \delta_{ij}\delta_{jk}$.

\begin{lemma} [Derivatives of $\Pi_i$ and $\Pi_i^*$]  \label{lem:Pi.derivatives} {\ }
\begin{enumerate}
\item[(i)] For $1 \leq i \leq n$ and $1 \leq j \leq n - 1$, we have
\begin{equation*}
\begin{split}
\frac{\partial \Pi_i(\theta, \theta')}{\partial \theta_j} &= \Pi_i(\theta, \theta') \left( \delta_{ij} - \Pi_j(\theta, \theta')\right), \\
\frac{\partial \Pi_i(\theta, \theta')}{\partial \theta_j'} &= -\Pi_i(\theta, \theta') \left( \delta_{ij} - \Pi_j(\theta, \theta')\right) \\
&\quad + \Pi_i(\theta, \theta') \left( \frac{1}{\ppi_i(\theta')} \frac{\partial \ppi_i}{\partial \theta_j'}(\theta') - \sum_{\ell = 1}^n \Pi_{\ell}(\theta, \theta') \frac{1}{\ppi_{\ell}(\theta')}\frac{\partial \ppi_{\ell}}{\partial \theta_{j}'}(\theta')\right). 
\end{split}
\end{equation*}
\item[(ii)] For $1 \leq i \leq n$ and $1 \leq j \leq n - 1$, we have
\begin{equation*}
\begin{split} 
\frac{\partial \Pi_i^*(\phi, \phi')}{\partial \phi_j} &= \Pi_i^*(\phi, \phi')(\delta_{ij} - \Pi_j^*(\phi, \phi')) \\
  &\quad + \Pi_i^*(\phi, \phi') \left(\frac{1}{\ppi_i(\phi)} \frac{\partial \ppi_i}{\partial \phi_j}(\phi) - \sum_{\ell = 1}^n \Pi_{\ell}^*(\phi, \phi') \frac{1}{\ppi_{\ell}(\phi)} \frac{\partial \ppi_{\ell}}{\partial \phi_j}(\phi)\right), \\
\frac{\partial \Pi_i(\phi, \phi')}{\partial \phi_j'} &= -\Pi_i^*(\phi, \phi')(\delta_{ij} - \Pi_j^*(\phi, \phi')).
\end{split}
\end{equation*}
\end{enumerate}
\end{lemma}

We also note the following easy fact which is used several times below. Since $\sum_{i = 1}^n \Pi_i(\theta, \theta') \equiv \sum_{i = 1}^n \Pi_i^*(\phi, \phi') \equiv 1$, the partial derivatives of $\sum_i \Pi_i$ and $\sum_i \Pi_i^*$ are all zero.

\begin{lemma} [Derviatives of $\ppi(\cdot)$]  \label{lem:pi.derivatives}
For $1 \leq i \leq n$ and $1 \leq j \leq n - 1$, we have
\begin{equation} \label{eqn:pi.derivative}
\begin{split}
\frac{\partial \ppi_i}{\partial \theta_j}(\theta) &= \ppi_i(\theta)(\delta_{ij} - \ppi_j(\theta)) - \ppi_i(\theta) \left( \frac{\partial \phi_i}{\partial \theta_j}(\theta) - \sum_{\ell = 1}^{n-1} \ppi_{\ell}(\theta) \frac{\partial \phi_{\ell}}{\partial \theta_j}(\theta)\right), \\
\frac{\partial \ppi_i}{\partial \phi_j}(\phi) &= -\ppi_i(\phi)(\delta_{ij} - \ppi_j(\phi)) + \ppi_i(\phi) \left(\frac{\partial \theta_i}{\partial \phi_j}(\phi) - \sum_{\ell = 1}^{n-1} \ppi_{\ell}(\phi) \frac{\partial \theta_{\ell}}{\partial \phi_j}(\phi) \right).
\end{split}
\end{equation} 
\end{lemma}
\begin{proof}
We prove the second formula and the proof of the first is similar. Using \eqref{eqn:portfolio.transport}, we write
\[
\ppi_i(\phi) = \frac{e^{\theta_i - \phi_i}}{\sum_{\ell = 1}^n e^{\theta_{\ell} - \phi_{\ell}}}
\]
and regard $\theta$ as a function of $\phi$ (recall that $\theta_n = \phi_n = 0$). Then
\begin{equation*}
\begin{split}
\frac{\partial \ppi_i}{\partial \phi_j}(\phi) &= \frac{e^{\theta_i - \phi_i} \left(\frac{\partial \theta_i}{\partial \phi_j}(\phi) -  \delta_{ij}\right)}{\sum_{\ell = 1}^n e^{\theta_{\ell} - \phi_{\ell}}} - \frac{e^{\theta_i - \phi_i}}{\left(\sum_{\ell = 1}^n e^{\theta_{\ell} - \phi_{\ell}}\right)^2} \sum_{\ell = 1}^n e^{\theta_{\ell} - \phi_{\ell}} \left( \frac{\partial \theta_{\ell}}{\partial \phi_j}(\phi) - \delta_{\ell j}\right) \\
&= -\ppi_i(\phi)(\delta_{ij} - \ppi_j(\phi)) + \ppi_i(\phi) \left( \frac{\partial \theta_i}{\partial \phi_j}(\phi) - \sum_{\ell = 1}^{n-1} \ppi_{\ell}(\phi) \frac{\partial \theta_{\ell}}{\partial \phi_j}(\phi) \right).
\end{split}
\end{equation*}
Note that the $n$th term of the sum is omitted because $\theta_n = 0$.
\end{proof}

Thanks to these formulas, computations in the primal and dual coordinates are very similar except for a change of sign. In the following we will often give details for one coordinate system and leave the other one to the reader.

Last but not least, let $\frac{\partial \phi}{\partial \theta}(\theta) = \left( \frac{\partial \phi_i}{\partial \theta_j}(\theta)\right)$ be the Jacobian of the change of coordinate map $\theta \mapsto \phi$. Similarly, we let
$\frac{\partial \theta}{\partial \phi}(\phi) = \left( \frac{\partial \theta_i}{\partial \phi_j}(\phi) \right)$ be the Jacobian of the inverse map $\phi \mapsto \theta$. The two Jacobians are inverses of each other, i.e.,
\begin{equation} \label{eqn:jacobian.identity}
\frac{\partial \phi}{\partial \theta}(\theta) \frac{\partial \theta}{\partial \phi}(\phi)  = I,
\end{equation}
where $I$ is the identity matrix. We denote the transpose of a matrix $A$ by $A^{\top}$.

\subsection{Riemannian metric}
For intuition, we first compute the Riemannian inner product using Euclidean coordinates. We let $T_p\Delta_n$ be the tangent space at $p$.

\begin{proposition} \label{prop:metric}
Let $u, v \in T_p\Delta_n$ be represented in Euclidean coordinates, i.e., $u = (u_1, \ldots, u_n) \in {\mathbb{R}}^n$ and $u_1 + \cdots + u_n = 0$, and similarly for $v$. Then
\begin{equation} \label{eqn:metric.euclidean}
\begin{split}
\langle u, v \rangle &= u^{\top} \left(-\Hess \varphi(p) - \nabla \varphi(p) \nabla \varphi(p)^{\top} \right) v\\
  &= u^{\top} \left( \frac{-1}{\Phi(p)} \Hess \Phi(p)\right) v,
\end{split}
\end{equation}
where $\Phi = e^{\varphi}$. 
\end{proposition}
\begin{proof}
By \cite[Proposition 11.3.1]{CU14} we have $\|v\|^2 = \left.\frac{d^2}{dt^2} T\left(p + tv \mid p\right)\right|_{t = 0}$, where
\[
T\left(p + tv \mid p\right) = \log\left(1 + t\nabla \varphi(p) \cdot v\right) - \left(\varphi(p + tv) - \varphi(p)\right).
\]
Differentiating two times and setting $t = 0$ gives the first equality in \eqref{eqn:metric.euclidean} when $u = v$, and polarizing gives the general case. The second equality follows from the chain rule.
\end{proof}

\begin{theorem}[Riemannian metric] \label{thm:Riemannian.metric} {\ }
\begin{enumerate}
\item[(i)] Under the primal coordinate system, the Riemannian metric is given by
\begin{equation} \label{eqn:g.primal}
\begin{split}
g_{ij}(\theta) &= \ppi_i(\theta)(\delta_{ij} - \ppi_j(\theta)) - \frac{\partial \ppi_i}{\partial \theta_j}(\theta).
\end{split}
\end{equation}
Its inverse is given by
\begin{equation} \label{eqn:g.primal.inverse}
g^{ij}(\theta) = \frac{1}{\ppi_j(\theta)} \frac{\partial \theta_i}{\partial \phi_j}(\phi) + \frac{1}{\ppi_n(\theta)} \sum_{\ell = 1}^{n - 1} \frac{\partial \theta_i}{\partial \phi_{\ell}}(\phi).
\end{equation}
\item[(ii)] Under the dual coordinate system, the Riemannian metric is given by
\begin{equation} \label{eqn:g.dual}
\begin{split}
g_{ij}^*(\phi) = \ppi_i(\phi)(\delta_{ij} - \ppi_j(\phi)) + \frac{\partial \ppi_i}{\partial \phi_j}(\phi).
\end{split}
\end{equation}
Its inverse is given by
\begin{equation} \label{eqn:g.dual.inverse}
g^{*ij}(\phi) = \frac{1}{\ppi_j(\phi)} \frac{\partial \phi_i}{\partial \theta_j}(\theta) + \frac{1}{\ppi_n(\phi)} \sum_{\ell = 1}^{n-1} \frac{\partial \phi_i}{\partial \theta_{\ell}}(\theta).
\end{equation}
\end{enumerate}
\end{theorem}
\begin{proof}
(i) By Lemma \ref{lem:coord.rep} and Lemma \ref{lem:Pi.derivatives}, we compute
\begin{align}
\frac{\partial}{\partial \theta_i} T \left(\theta \mid \theta'\right) &= \Pi_i(\theta, \theta') - \ppi_i(\theta), \label{eqn:T.partial.i} \\ 
\frac{\partial}{\partial \theta_i'} T \left(\theta \mid \theta'\right) &= -\Pi_i(\theta, \theta') + \ppi_i(\theta') + \sum_{\ell = 1}^n \Pi_{\ell}(\theta, \theta') \frac{1}{\ppi_{\ell}(\theta')} \frac{\partial \ppi_{\ell}}{\partial \theta_i'}(\theta'). \label{eqn:T.partial.i.prime}
\end{align}
Differentiating \eqref{eqn:T.partial.i} with respect to $\theta_j'$, we have
\begin{equation} \label{eqn:T.second.derivative}
\begin{split}
\frac{\partial^2}{\partial \theta_i \partial \theta_j'} T\left(\theta \mid \theta'\right) &= -\Pi_i(\theta, \theta') \left(\delta_{ij} - \Pi_j(\theta, \theta')\right) \\
& \quad + \Pi_i(\theta, \theta')  \left( \frac{1}{\ppi_i(\theta')}\frac{\partial \ppi_i}{\partial \theta_j'}(\theta') - \sum_{\ell = 1}^n \Pi_{\ell}(\theta, \theta') \frac{1}{\ppi_{\ell}(\theta')}\frac{\partial \ppi_{\ell}}{\partial \theta_j'}(\theta')\right). 
\end{split}
\end{equation}
Setting $\theta = \theta'$, we get $g_{ij}(\theta) = \ppi_i(\theta)(\delta_{ij} - \ppi_j(\theta)) - \frac{\partial \ppi_i}{\partial \theta_j}(\theta)$.

By Lemma \ref{lem:pi.derivatives}, we have the alternative expression
\begin{equation} \label{eqn:gij.alternative}
g_{ij}(\theta) = \ppi_i(\theta) \left( \frac{\partial \phi_i}{\partial \theta_j}(\theta) - \sum_{\ell = 1}^{n-1} \ppi_{\ell}(\theta) \frac{\partial \phi_{\ell}}{\partial \theta_j}(\theta)\right).
\end{equation}
Expressing \eqref{eqn:gij.alternative} in matrix form, we have
\begin{equation} \label{eqn:g.matrix}
\left(g_{ij}(\theta)\right) = \diag(\ppi(\theta)) (I - {\bf 1}\ppi^{\top}(\theta)) \frac{\partial \phi}{\partial \theta}(\theta),
\end{equation}
where here $\ppi(\theta) = (\ppi_1(\theta), \ldots, \ppi_{n-1}(\theta))^{\top}$, ${\bf 1} = {\bf 1}_{n-1} = (1, \ldots, 1)^{\top}$ and $I = I_{n-1}$ is the identity matrix.

To invert \eqref{eqn:g.matrix} we use the fact that
\[
(I - {\bf 1} \ppi^{\top}(\theta))^{-1} = I + \frac{{\bf 1} \ppi^{\top}(\theta)}{\ppi_n(\theta)}.
\]
This can be verified directly or seen as a special case of the Sherman-Morrison formula. Using \eqref{eqn:jacobian.identity}, we have
\[
\left(g^{ij}(\theta)\right) = \frac{\partial \theta}{\partial \phi}(\phi) \left(I + \frac{{\bf 1} \ppi^{\top}(\theta)}{\ppi_n(\theta)}\right) \diag\left(\frac{1}{\ppi(\theta)}\right).
\]
Now \eqref{eqn:g.primal.inverse} follows by expanding the matrix product.

(ii) The proofs of \eqref{eqn:g.dual} and \eqref{eqn:g.dual.inverse} follow the same lines. For later use we record the following formulas:
\begin{align}
\frac{\partial}{\partial \phi_i} T\left( \phi \mid \phi'\right) &= \Pi_i^*(\phi, \phi') - \ppi_i(\phi) + \sum_{\ell = 1}^n \frac{1}{\ppi_{\ell}(\phi)}\frac{\partial \ppi_{\ell}}{\partial \phi_i}(\phi) \Pi_{\ell}^*(\phi, \phi'), \label{eqn:T.derivative.phi.1}\\
\frac{\partial}{\partial \phi_i'} T\left( \phi \mid \phi'\right) &= -\Pi_i^*(\phi, \phi') + \ppi_i(\phi'), \label{eqn:T.derivative.phi.2}
\end{align}
\begin{equation} \label{eqn:T.sec.derivative}
\begin{split}
\frac{\partial^2}{\partial \phi_i \partial \phi_j'} T\left( \phi \mid \phi'\right) &= -\Pi_j^*(\phi, \phi')(\delta_{ij} - \Pi_i^*(\phi, \phi')) \\ 
&\quad - \Pi_j^*(\phi, \phi') \left(\frac{1}{\ppi_j(\phi)} \frac{\partial \ppi_j}{\partial \phi_i}(\phi) - \sum_{\ell = 1}^n \frac{1}{\ppi_{\ell}(\phi)} \frac{\partial \ppi_{\ell}}{\partial \phi_i}(\phi) \Pi_{\ell}^*(\phi, \phi')\right).
\end{split}
\end{equation}
\end{proof}

\begin{remark}
By Lemma \ref{lem:weight.as.derivative} we have
\begin{equation} \label{eqn:pi.symmetry}
\frac{\partial \ppi_i}{\partial \theta_j}(\theta) = \frac{\partial^2}{\partial \theta_i \partial \theta_j} f(\theta) = \frac{\partial \ppi_j}{\partial \theta_i}(\theta).
\end{equation}
Thus the right hand side of \eqref{eqn:g.primal} is symmetric in $i$ and $j$ despite its appearence. Similarly, we have $\frac{\partial \ppi_i}{\partial \phi_j} = \frac{\partial \ppi_j}{\partial \phi_i}$.
\end{remark}

\subsection{Primal and dual connections}
\begin{theorem} [Primal and dual connections] \label{thm:induced.connections} \label{thm:connection}{\ }
\begin{enumerate}
\item[(i)]  Under the primal coordinate system, the coefficients of the primal connection $\nabla$ is given by
\begin{eqnarray}
\Gamma_{ijk}(\theta) &=& \delta_{ij} g_{ik}(\theta) - \ppi_i(\theta) g_{jk}(\theta) - \ppi_j(\theta) g_{ik}(\theta), \label{eqn:fgp.primal.connection}\\
\Gamma_{ij}^k(\theta) &=& \delta_{ijk} - \delta_{ik} \ppi_j(\theta) - \delta_{jk} \ppi_i(\theta). \label{eqn:fgp.primal.connection.2}
\end{eqnarray}
\item[(ii)] Under the dual coordinate system, the coefficients of the dual connection $\nabla^*$ is given by
\begin{eqnarray}
\Gamma_{ijk}^*(\phi) &=& -\delta_{ij} g_{ik}^*(\phi) + \ppi_{i}(\phi) g_{jk}^*(\phi) + \ppi_{j}(\phi) g_{ik}^*(\phi),\\ \label{eqn:fgp.dual.connection}
\Gamma_{ij}^{*k}(\phi) &=& -\delta_{ijk} + \delta_{ik} \ppi_{j}(\phi) + \delta_{jk} \ppi_{i}(\phi). \label{eqn:fgp.dual.connection.2}
\end{eqnarray}
\end{enumerate}
\end{theorem}
\begin{proof}
We prove (ii) and leave (i) to the reader. By \eqref{eqn:T.sec.derivative}, we have
\begin{equation*}
\begin{split}
\frac{\partial^2}{\partial \phi_k \partial \phi_i'} T\left( \phi \mid \phi'\right) &= -\Pi_i^*(\phi, \phi')(\delta_{ik} - \Pi_k^*(\phi, \phi')) \\
&\quad - \Pi_i^*(\phi, \phi') \left(\frac{1}{\ppi_i(\phi)} \frac{\partial \ppi_i}{\partial \phi_k}(\phi) - \sum_{\ell = 1}^n \frac{1}{\ppi_{\ell}(\phi)} \frac{\partial \ppi_{\ell}}{\partial \phi_k}(\phi) \Pi_{\ell}(\phi, \phi')\right).
\end{split}
\end{equation*}
For notational convenience we momentarily suppress $\phi$ and $\phi'$ in the computation (later we will do so without comment). Differentiating one more time, we have
\begin{equation*}
\begin{split}
& \frac{\partial^3}{\partial \phi_k \partial \phi_i' \partial \phi_j'} T\left(\phi \mid \phi'\right) \\
&= -\delta_{ik} \frac{\partial \Pi_i^*}{\partial \phi_j'} + \Pi_i^* \frac{\partial \Pi_k^*}{\partial \phi_j'} + \Pi_k^* \frac{\partial \Pi_i^*}{\partial \phi_j'} \\
&\quad - \frac{\partial \Pi_i^*}{\partial \phi_j'} \left( \frac{1}{\ppi_i}\frac{\partial \ppi_i}{\partial \phi_k} - \sum_{\ell = 1}^n \frac{1}{\ppi_{\ell}} \frac{\partial \ppi_{\ell}}{\partial \phi_k} \Pi_{\ell}^*\right)  + \Pi_i \sum_{\ell = 1}^n \frac{1}{\ppi_{\ell}} \frac{\partial \ppi_{\ell}}{\partial \phi_k} \frac{\partial \Pi_{\ell}^*}{\partial \phi_j'} \\
&= \delta_{ik} \Pi_i^*(\delta_{ij} - \Pi_j^*) - \Pi_i^* \Pi_k^*(\delta_{jk} - \Pi_j^*) - \Pi_k^* \Pi_i^*(\delta_{ij} - \Pi_j^*) \\
&\quad + \Pi_i^*(\delta_{ij} - \Pi_j^*)  \left( \frac{1}{\ppi_i}\frac{\partial \ppi_i}{\partial \phi_k} - \sum_{\ell = 1}^n \frac{1}{\ppi_{\ell}} \frac{\partial \ppi_{\ell}}{\partial \phi_k} \Pi_{\ell}^*\right) - \Pi_i^* \sum_{\ell = 1}^n \frac{1}{\ppi_{\ell}}\frac{\partial \ppi_{\ell}}{\partial \phi_k} \Pi_{\ell}^*(\delta_{\ell j} - \Pi_j^*).
\end{split}
\end{equation*}
Evaluating at $\phi = \phi'$ and simplifying, we get
\begin{equation} \label{eqn:gamma.star.intermediate}
\begin{split}
\Gamma_{ijk}^*(\phi) &= -\delta_{ijk} \ppi_i - 2\ppi_i\ppi_j\ppi_k + \delta_{ij}\ppi_i \ppi_k + \delta_{jk} \ppi_j \ppi_i + \delta_{ki} \ppi_k \ppi_j \\
  &\quad \quad - \delta_{ij} \frac{\partial \ppi_i}{\partial \phi_k} + \ppi_j \frac{\partial \ppi_i}{\partial \phi_k} + \ppi_i \frac{\partial \ppi_j}{\partial \phi_k}.
\end{split}
\end{equation}
By \eqref{eqn:g.dual}, we have $\frac{\partial \ppi_i}{\partial \phi_j} = g_{ij}^* - \ppi_i(\delta_{ij} - \ppi_j)$. Plugging this into \eqref{eqn:gamma.star.intermediate} and simplifying, we have $\Gamma_{ijk}^*(\phi) = -\delta_{ij} g_{ik}^*(\phi) + \ppi_i(\phi) g_{jk}^*(\phi) + \ppi_j(\phi) g_{ik}^*(\phi)$. Finally, we compute
\begin{equation*}
\begin{split}
\Gamma_{ij}^{*k}(\phi) 
  &= \sum_{m = 0}^{n - 1} \left( -\delta_{ij} g_{im}^*(\phi) + \ppi_i(\phi) g_{jm}^*(\phi) + \ppi_j(\phi) g_{im}^*(\phi)\right) g^{*mk}(\phi) \\
  &= -\delta_{ijk} + \delta_{ik} \ppi_j(\phi) + \delta_{jk} \ppi_i(\phi).   
\end{split}
\end{equation*}
\end{proof}

\begin{remark}
It is interesting to note that although the connections are defined in terms of the third order derivatives of $T\left(\cdot \mid \cdot\right)$, the coefficients $\Gamma_{ij}^k(\theta)$ and $\Gamma_{ij}^{*k}(\phi)$ are given in terms of the portfolio $\ppi$ which is a normalized gradient of $\varphi$.
\end{remark}

\subsection{Curvatures}
It is well known (see \cite[Chapter 1]{A16}) that the induced geometric structure of any Bregman divergence is dually flat. This is not the case for the geometry of L-divergence whenever $n \geq 3$ (when $n = 2$ the simplex $\Delta_2$ is one-dimensional). To verify this we compute the Riemann-Christoffel curvature tensors of the primal and dual connections. In this (and only this) subsection we adopt the Einstein summation notation (see \cite[p.20]{A16}) to avoid writing a lot of summation signs.

The Riemann-Christoffel (RC) curvature tensor of a connection $\nabla$ is defined for smooth vector fields $X$, $Y$ and $Z$ by
\[
R(X, Y)Z = \nabla_X (\nabla_Y Z) - \nabla_Y (\nabla_X Z) - \nabla_{[X, Y]} Z,
\]
where $[X, Y]$ is the Lie bracket. Given a coordinate system $\xi$, its coefficients are given by $R_{ijk}^{\ell}$ which satisfy $R\left( \frac{\partial}{\partial \xi_i}, \frac{\partial}{\partial \xi_j}\right) \frac{\partial}{\partial \xi_k} = \sum_{\ell} R_{ijk}^{\ell} \frac{\partial}{\partial \xi_{\ell}}$. By \cite[(5.66)]{A16}, we have
\[
R_{ijk}^{\ell} = \frac{\partial}{\partial \theta_i} \Gamma_{jk}^{\ell} - \frac{\partial}{\partial \theta_j} \Gamma_{ik}^{\ell} + \Gamma_{im}^{\ell} \Gamma_{jk}^m - \Gamma_{jm}^{\ell} \Gamma_{ik}^m.
\]

\begin{theorem} [Primal and dual Riemann-Christoffel curvatures] \label{thm:RC.curvature}
Let $R$ and $R^*$ be the RC curvature tensors of the primal and dual connections respectively.
\begin{enumerate}
\item[(i)] In primal coordinates, the coefficients of $R$ are given by
\begin{equation} \label{lem:RC.primal}
R_{ijk}^{\ell}(\theta) = \delta_{\ell j} g_{ik}(\theta) - \delta_{\ell i} g_{jk}(\theta).
\end{equation}
\item[(ii)] In dual coordinates, the coefficients of $R^*$ are given by
\begin{equation} \label{lem:RC.dual}
R_{ijk}^{*\ell}(\phi) = \delta_{\ell j} g_{ik}^*(\phi) - \delta_{\ell i} g_{jk}^*(\phi).
\end{equation}
\end{enumerate}
In particular, for $n \geq 3$ both $R$ and $R^*$ are nonzero everywhere on $\Delta_n$.
\end{theorem}
\begin{proof}
We prove the statements for $R$. Using \eqref{eqn:fgp.primal.connection} and suppressing the argument, we have
\[
\frac{\partial}{\partial \theta_i} \Gamma_{jk}^{\ell} = -\delta_{\ell j} \frac{\partial \ppi_k}{\partial \theta_i} - \delta_{\ell k} \frac{\partial \ppi_j}{\partial \theta_i}, \quad
\frac{\partial}{\partial \theta_j} \Gamma_{ik}^{\ell} = -\delta_{\ell i} \frac{\partial \ppi_k}{\partial \theta_j} - \delta_{\ell k} \frac{\partial \ppi_i}{\partial \theta_j}.
\]
From \eqref{eqn:pi.symmetry} it follows that
\begin{equation} \label{eqn:difference.1}
\frac{\partial}{\partial \theta_i} \Gamma_{jk}^{\ell} - \frac{\partial}{\partial \theta_j} \Gamma_{ik}^{\ell} = -\delta_{\ell j} \frac{\partial \ppi_k}{\partial \theta_i} + \delta_{\ell i} \frac{\partial \ppi_k}{\partial \theta_j}.
\end{equation}

Next we compute (with some work)
\begin{equation} \label{eqn:difference.2}
\Gamma_{im}^{\ell} \Gamma_{jk}^m - \Gamma_{jm}^{\ell} \Gamma_{ik}^m = -\delta_{\ell i} \delta_{ jk} \ppi_j + \delta_{\ell j} \delta_{ik} \ppi_i + \delta_{\ell i} \ppi_j \ppi_k - \delta_{\ell j} \ppi_i \ppi_k.
\end{equation}

Combining \eqref{eqn:difference.1} and \eqref{eqn:difference.2}, we have
\begin{equation*}
\begin{split}
R_{ijk}^{\ell}(\theta) &= -\delta_{\ell j} \frac{\partial \ppi_k}{\partial \theta_i} + \delta_{\ell i} \frac{\partial \ppi_k}{\partial \theta_j} - \delta_{\ell i} \delta_{jk} \ppi_j + \delta_{\ell j} \delta_{ik} \ppi_i + \delta_{\ell i} \ppi_j \ppi_k - \delta_{\ell j} \ppi_i \ppi_k \\
  &= \delta_{\ell j} \left(\delta_{ik} \ppi_i - \ppi_i \ppi_k - \frac{\partial \ppi_k}{\partial \theta_i}\right) - \delta_{\ell i} \left(\delta_{jk} \ppi_j - \ppi_j \ppi_k - \frac{\partial \ppi_k}{\partial \theta_j}\right) \\
  &= \delta_{\ell j} g_{ik} - \delta_{\ell i} g_{jk}.
\end{split}
\end{equation*}

To see that $R$ does not vanish for $n \geq 3$, suppose on the contrary that $R(\theta)= 0$. Then $R_{ijk}^{\ell}(\theta) = \delta_{\ell j} g_{ik}(\theta) - \delta_{\ell i} g_{jk}(\theta) = 0$ for all values of $i, j, k, \ell$. Fix $i$ and $k$. Letting $\ell = j$, we have $g_{ik}(\theta) = \delta_{ij} g_{jk}(\theta)$. Next let $j \neq i$ (here we need $\dim \Delta_n = n - 1 \geq 2$). Then we get $g_{ik}(\theta) = 0$. Since $i$ and $k$ are arbitrary, we have $g(\theta) = 0$ which is a contradiction.
\end{proof}

We end this section by showing that the primal and dual connections have constant sectional curvature $-1$. See \cite[Definition 7.10.5]{CU14} for the definition of Ricci curvature.

\begin{corollary} [Primal and dual sectional curvatures] \label{cor:Ricci}
The primal and dual connections have constant sectional curvature $-1$ with respect to $g$. In particular, the primal and dual Ricci curvatures satisfy the Einstein condition
\[
\Ric = \Ric^* = -(n - 2)g.
\]
\end{corollary}
\begin{proof}
The primal connection $\nabla$ has constant sectional curvature $k$ with respect to $g$ if and only if
\[
R(X, Y)Z \equiv k(\langle Y, Z\rangle X - \langle X, Z \rangle Y)
\]
for all smooth vector fields $X$, $Y$ and $Z$ (see \cite[(9.7.41)]{CU14}). For the primal Riemann-Christoffel curvature tensor we have
\[
R\left(\frac{\partial}{\partial \theta_i}, \frac{\partial}{\partial \theta_j}\right)\frac{\partial}{\partial \theta_k} = R_{ijk}^{\ell} \frac{\partial}{\partial \theta_{\ell}} =  -  \left(\left\langle \frac{\partial}{\partial \theta_j}, \frac{\partial}{\partial \theta_k}\right\rangle \frac{\partial}{\partial \theta_i} - \left\langle \frac{\partial}{\partial \theta_i}, \frac{\partial}{\partial \theta_k}\right\rangle \frac{\partial}{\partial \theta_j}\right),
\]
which implies that the sectional curvature is $k = -1$. The claim for Ricci curvature follows immediately by taking trace (see for example \cite[(4.31)]{S12}). The proof for the dual curvatures is the same.
\end{proof}

\section{Geodesics and generalized Pythagorean theorem} \label{sec:geodesic}
Armed with the primal and dual connections we can formulate the primal and dual geodesic equations. Their solutions are the primal and dual geodesics which will be studied in this section. The highlight of this section is the generalized Pythagorean theorem (Theorem \ref{thm:pyth}). Along the way we will prove some remarkable properties of the geometric structure $(g, \nabla, \nabla^*)$.

\subsection{Primal and dual geodesics}
Note that in Figure \ref{fig:pyth} the primal geodesic is drawn as a straight line in $\Delta_n$. We now prove that this is indeed the case. The same is true for the dual geodesic in dual Euclidean coordinates.

Let $\gamma: [0, 1] \rightarrow \Delta_n$ be a smooth curve. We denote time derivatives by $\dot{\gamma}(t)$. Let $\theta(t)$ and $\phi(t)$ be the primal and dual coordinate representations of $\gamma$. We say that $\gamma$ is a primal geodesic if its satisfies
\[
\ddot{\theta}_k(t) + \sum_{i, j = 1}^{n-1} \Gamma_{ij}^k(\theta(t)) \dot{\theta}_i(t) \dot{\theta}_j(t) = 0, \quad k = 1, \ldots, n - 1.
\]
It is a dual geodesic if its satisfies
\[
\ddot{\phi}_k(t) + \sum_{i, j = 1}^{n-1} \Gamma_{ij}^{*k}(\phi(t)) \dot{\phi}_i(t) \dot{\phi}_j(t) = 0, \quad k = 1, \ldots, n - 1.
\]
By Theorem \ref{thm:induced.connections}, the primal geodesic equation in primal coordinates is
\begin{equation} \label{eqn:primal.geodesic.equation}
\ddot{\theta}_k(t) + 2 \dot{\theta}_k(t) \sum_{\ell = 1}^{n-1} \ppi_{\ell}(\theta(t)) \dot{\theta}_{\ell}(t) = 0, \quad k = 1, \ldots, n - 1.
\end{equation}
The dual geodesic equation in dual coordinates is
\begin{equation} \label{eqn:dual.geodesic.equation}
\ddot{\phi}_k(t) - 2 \dot{\phi}_k(t) \sum_{\ell = 1}^{n-1} \ppi_{\ell}(\phi(t)) \dot{\phi}_{\ell}(t) = 0, \quad k = 1, \ldots, n - 1.
\end{equation}

\begin{theorem}[Primal and dual geodesics] \label{thm:geodesics} {\ } 
\begin{enumerate}
\item[(i)] Let $\gamma: [0, 1] \rightarrow \Delta_n$ be a primal geodesic. Then the trace of $\gamma$ in $\Delta_n$ is the Euclidean straight line in $\Delta_n$ joining $\gamma(0)$ and $\gamma(1)$.
\item[(ii)] Let $\gamma^*: [0, 1] \rightarrow \Delta_n$ be a dual geodesic. For each $t$, let $p^*(t)$ be the dual Euclidean coordinate of $\gamma(t)$. Then the trace of $p^*$ in $\Delta_n$ is the Euclidean straight line in $\Delta_n$ joining $p^*(0)$ and $p^*(1)$.
\end{enumerate}
\end{theorem}

We will prove (i) and leave (ii) to the reader. The proof makes use of the following lemmas.

\begin{lemma} \label{lem:ode1}
Let $q, r \in \Delta_n$ and let $\theta^q$ and $\theta^r$ their primal coordinates respectively. Consider the differential equation
\begin{equation}\label{eqn:nonlinear.ode}
h''(t) - 2(h'(t))^2 \sum_{\ell = 1}^{n-1} \ppi_{\ell}(\theta(t)) \frac{\theta_{\ell}^r - \theta_{\ell}^q}{ (1 - h(t)) e^{\theta^q_{\ell}} + h(t) e^{\theta^r_{\ell}}} = 0,
\end{equation}
where $\theta(t)$ is defined by
\begin{equation} \label{eqn:primal.geodesic.candidate}
\theta_k(t) = \log\left( (1 - h(t)) e^{\theta^q_k} + h(t) e^{\theta^r_k}\right), \quad k = 1, \ldots, n - 1,
\end{equation}
Then there exists a unique solution $h: [0, 1] \rightarrow {\mathbb{R}}$ satisfying $h(0) = 0$, $h'(t) > 0$ for $t \in [0, 1]$, and $h(1) = 1$.
\end{lemma}
\begin{proof}
First we note that if $h(t)$ satisfies \eqref{eqn:nonlinear.ode}, then for any $c > 0$ the map $t \mapsto h(ct)$ also does (with domain scaled by $1/c$). Also, if $h'(t_0) = 0$ then $h(t) = h(t_0)$ for all $t \geq t_0$. Let $h_0(t)$ be a maximal solution to \eqref{eqn:nonlinear.ode} defined on an interval $[0, t_{\max})$ with $h_0(0) = 0$ and $h_0'(0) > 0$. Since the equation can be solved in a neighborhood of $t = 0$, we have $t_{\max} > 0$. By the previous remark, $h_0$ is strictly increasing on $[0, t_{\max})$. If $h_0(t)$ hits $1$ at some $t = t_0 < t_{\max}$, the function $h(t) = h_0(t/t_0)$ is a solution with the desired properties. In fact, we claim that
\begin{equation} \label{eqn:sup.M}
\lim_{t \uparrow t_{\max}} h_0(t) = \sup_{t < t_{\max}} h_0(t) = M,
\end{equation}
where $M$ is defined by
\begin{equation*}
\begin{split}
M &:= \inf\left\{s > 0:   \min_{1 \leq \ell \leq n - 1} (1 - s) e^{\theta^q_{\ell}} + s e^{\theta^r_{\ell}} = 0\right\} \\
&= \min_{1 \leq \ell \leq n - 1}  \left(\frac{e^{\theta_{\ell}^q}}{e^{\theta_{\ell}^q} - e^{\theta_{\ell}^r}}1_{\{\theta^q_{\ell} > \theta^r_{\ell}\}} +   \infty \cdot 1_{\{\theta^q_{\ell} \leq \theta^r_{\ell}\}} \right) > 1.
\end{split}
\end{equation*}
Thus as $h(t)$ approaches $M$ from below, at least one of the fractions in \eqref{eqn:nonlinear.ode} blows up to $+\infty$. It follows that $M' := \sup_{t < t_{\max}} h(t) \leq M$. Suppose on the contrary that $M' < M$. Let $h_1(t)$, $t \in (-\epsilon, \epsilon)$ be a solution to \eqref{eqn:nonlinear.ode} satisfying $h_1(0) = M'$ and $h_1'(0) > 0$. Note that $h_1$ exists because the fractions in \eqref{eqn:nonlinear.ode} are finite and continuous near $M' < M$. Then the range of $h_1$ contains an open interval containing $M'$. Thus there exists $t_0 < t_{\max}$, $c > 0$ and $t_1 < 0$ such that $ct_1 > -\epsilon$, $h_0(t_0) = h_1(ct_1)$ and $h_0'(t_0) = \left.\frac{d}{dt} h_1(ct)\right|_{t = t_1}$. This allows us to extend the range of $h_0$ beyond $M'$ which contradicts the maximality of $M'$. By the uniqueness theorem for ODE (note that \eqref{eqn:nonlinear.ode} has smooth coefficients), the solution $h$ is unique.
\end{proof}

\begin{lemma} \label{lem:ode2}
Let $h$ be the solution in Lemma \ref{lem:ode1}, and consider the curve $\gamma: [0, 1] \rightarrow \Delta_n$ given in exponential coordinates by \eqref{eqn:primal.geodesic.candidate}. Then $\gamma$ is a primal geodesic from $q$ to $r$. Moreover, the trace of $\gamma$ in $\Delta_n$ is the Euclidean straight line in $\Delta_n$ joining $q = \gamma(0)$ and $r = \gamma(1)$.
\end{lemma}
\begin{proof}
That $\gamma$ is a primal geodesic from $q$ to $r$ can be verified directly by differentiating \eqref{eqn:primal.geodesic.candidate} and plugging into the primal geodesic equation \eqref{eqn:primal.geodesic.equation}. We omit the computational details.

To see that the trace of $\gamma$ is a Euclidean straight line in $\Delta_n$, consider its Euclidean representation $p(t) = \left(p_1(t), \ldots, p_n(t)\right)$. By \eqref{eqn:primal.geodesic.candidate} we have
\[
e^{\psi(\theta(t))} = \frac{1}{p_n(t)} = (1 - h(t)) e^{\psi(\theta^q)} + h(t) e^{\psi(\theta^r)}.
\]
Solving for $h(t)$ gives
\begin{equation} \label{eqn:psi.interpolation}
h(t) = \frac{e^{\psi(\theta(t))} - e^{\psi(\theta^q)}}{e^{\psi(\theta^r)} - e^{\psi(\theta^q)}}.
\end{equation}
Expressing \eqref{eqn:primal.geodesic.candidate} in Euclidean coordinates and using \eqref{eqn:psi.interpolation}, we get after some algebra that
\[
p_k(t) = e^{\theta_k^q} p_n(t) + \left(1 - e^{\psi(\theta^q)} p_n(t)\right) \frac{e^{\theta^r_k} - e^{\theta^q_k}}{e^{\psi(\theta^r)} - e^{\psi(\theta^q)}}, \quad k = 1, \ldots, n - 1.
\]
Hence there exists $a_k, b_k$ such that
\begin{equation} \label{eqn:Euclidean.line}
p_k(t) = a_k + b_k p_n(t), \quad k = 1, \ldots, n - 1.
\end{equation}
Together with the identity $p_1(t) + \cdots + p_n(t) \equiv 1$, \eqref{eqn:Euclidean.line} shows that $\gamma$ is a time change of the Euclidean straight line from $q$ to $r$.
\end{proof}

\begin{proof}[Proof of Theorem \ref{thm:geodesics}]
(i) We have shown in Lemma \ref{lem:ode2} that for any pair of points $(q, r)$ in $\Delta_n$, there exists a primal geodesic from $q$ to $r$ which is a time-changed Euclidean straight line. It remains to observe that the geodesic is unique. Indeed, let $\gamma$ be any primal geodesic. Then it solves the primal geodesic equation \eqref{eqn:primal.geodesic.equation}. Consider the initial velocity $\gamma'(0)$. Let $q = \gamma(0)$. By varying $r \in \Delta_n$ as well as the initial speed, there exists a primal geodesic $\widetilde{\gamma}$ in the form \eqref{eqn:primal.geodesic.candidate} from $q$ to $r$ such that $\dot{\gamma}(0) = \dot{\widetilde{\gamma}}(0)$. (This is because $T_q \Delta_n = \bigcup_{r \in \Delta_n, \kappa > 0} \kappa(r - q)$.) By the uniqueness theorem of ODE we have $\gamma(\cdot) = \widetilde{\gamma}(\cdot)$.

Using the dual coordinate system $\phi$ and dual Euclidean coordinate system $p^*$, (ii) can be proved in a similar way by considering the curve defined by
\begin{equation} \label{eqn:dual.geodesic.curve}
\phi_k(t) = \log \left(\frac{1}{(1 - h(t)) e^{-\phi_k^q} + h(t) e^{-\phi_k^p}}\right), \quad k = 1, \ldots, n - 1. %
\end{equation}
\end{proof}

A connection is projectively flat if there is a coordinate system under which the geodesics are straight lines up to time reparameterizations. We say that the geometric structure $(g, \nabla, \nabla^*)$ is dually projectively flat if both $\nabla$ and $\nabla^*$ are projectively flat. In view of Theorem \ref{thm:geodesics} we have the following corollary.

\begin{corollary} \label{cor:dual.projectively.flat}
The manifold $\Delta_n$ equipped with the geometric structure $(g, \nabla, \nabla^*)$ is dually projectively flat, but is not flat for $n \geq 3$.
\end{corollary}

\subsection{Gradient flows and inverse exponential maps}
Motivated by the recent paper \cite{AA15} we relate the primal and dual geodesics with gradient flows under the L-divergence. Fix $p, q, r \in \Delta_n$. Consider the following gradient flows starting at $q$:
\begin{equation} \label{eqn:primal.flow}
\begin{cases}
\dot{\gamma}(t) = -\grad \ T\left(r \mid  \cdot\right)(\gamma(t)) \\
\gamma(0) = q
\end{cases}
\quad \text{(primal flow)}
\end{equation}
and
\begin{equation} \label{eqn:dual.flow}
\begin{cases}
\dot{\gamma}^*(t) = -\grad \ T\left(\cdot \mid p\right)(\gamma^*(t)) \\
\gamma^*(0) = q
\end{cases}
\quad \text{(dual flow)}
\end{equation}
Here $\grad$ denotes the Riemannian gradient with respect to the metric $g$. We call \eqref{eqn:primal.flow} the primal flow and \eqref{eqn:dual.flow} the dual flow.

It can be verified easily that
\[
\quad \frac{d}{dt} T(r \mid \gamma(t)) = -\| \dot{\gamma}(t)\|^2 \quad \text{and} \quad
\frac{d}{dt} T(\gamma^*(t) \mid p) = -\| \dot{\gamma}^*(t)\|^2.
\]
Since $T\left(q \mid p\right) = 0$ if and only if $p = q$, by standard ODE theory it can be shown that the solutions $\gamma(t)$ and $\gamma^*(t)$ are defined for $t \in [0, \infty)$ and
\[
\lim_{t \rightarrow \infty} \gamma(t) = r, \quad \lim_{t \rightarrow \infty} \gamma^*(t) = p.
\]
In other words, both gradient flows converge to the unique minimizers. 

\begin{theorem} [Gradient flows] \label{thm:gradient.flows} {\ }
\begin{enumerate}
\item[(i)] The primal flow $\gamma(t)$ is a time change of the primal geodesic from $q$ to $r$.
\item[(ii)] The the dual flow $\gamma^*(t)$ is a time change of the dual geodesic from $q$ to $p$.
\end{enumerate}
\end{theorem}

Recall the concept of exponential map. For $q \in \Delta_n$ and $v \in T_q\Delta_n$, consider the primal geodesic $\gamma$ starting at $q$ with initial velocity $v$. If $\gamma$ is defined up to time $1$, we define $\exp_q(v) = \gamma(1)$. The dual exponential map $\exp^*$ is defined analogously. As a corollary of Theorems \ref{thm:geodesics} and \ref{thm:gradient.flows} we have the following characterization of the primal and dual inverse exponential maps.

\begin{corollary}[Inverse exponential maps] 
Let $\exp$ and $\exp^*$ be the exponential maps with respect to the primal and dual connections respectively. For $p, q \in \Delta_n$ we have
\begin{enumerate}
\item[(i)] $\exp_q^{-1}(p) \propto -\grad \ T\left(p \mid \cdot\right)(q)$.
\item[(ii)] $\left(\exp_q^*\right)^{-1}(p) \propto -\grad\ T\left(\cdot \mid p\right)(q)$.
\end{enumerate}
\end{corollary}

To prove Theorem \ref{thm:gradient.flows} we begin by computing the Riemannian gradients of $T\left(r \mid \cdot\right)$ and $T\left(\cdot \mid p\right)$. The computation is somewhat tricky and will be given in the Appendix. Recall the notations in \eqref{eqn:Pi}. 

\begin{lemma} [Riemannian gradients] \label{lem:gradient}
Let $p, q, r \in \Delta_n$.
\begin{enumerate}
\item[(i)] Under the primal coordinate system, we have
\begin{equation} \label{eqn:Riemannian.gradient.primal}
\begin{split}
\grad \ T\left(r \mid \cdot\right)(q) &= \sum_{i = 1}^{n-1} \left(-\frac{\ppi_i(\theta^r, \theta^q)}{\ppi_i(\theta^q)} + \frac{\ppi_n(\theta^r, \theta^q)}{\ppi_n(\theta^q)}\right) \frac{\partial}{\partial \theta_i^q} \\
  &= \frac{1}{\sum_{\ell = 1}^n \ppi_{\ell}(\theta^q) e^{\theta_{\ell}^r - \theta_{\ell}^q}} \sum_{i = 1}^{n-1} \left(-e^{\theta_i^r - \theta_i^q} + 1\right) \frac{\partial}{\partial \theta_i^q}.
\end{split}
\end{equation}
\item[(ii)] Under the dual coordinate system, we have
\begin{equation} \label{eqn:Riemannian.gradient.dual}
\begin{split}
\grad \ T\left(\cdot \mid p\right)(q) &= \sum_{i = 1}^{n-1} \left( \frac{\ppi_i^*(\phi^q, \phi^p)}{\ppi_i(\phi^q)} - \frac{\ppi_n^*(\phi^q, \phi^p)}{\ppi_n(\phi^q)}\right) \frac{\partial}{\partial \phi_i^q} \\
  &= \frac{1}{\sum_{\ell = 1}^n \ppi_{\ell}(\phi^q) e^{\phi_{\ell}^q - \phi_{\ell}^p}} \sum_{i = 1}^{n-1} \left( e^{\phi_i^q - \phi_i^p} - 1\right) \frac{\partial}{\partial \phi_i^q}.
\end{split}
\end{equation}
\end{enumerate}
\end{lemma}

\begin{proof}[Proof of Theorem \ref{thm:gradient.flows}]
We prove (i) and leave (ii) to the reader. Let $\theta(t)$ be the primal representation of the primal flow starting at $q$. By Lemma \ref{lem:gradient}, at any time $t$ we have
\[
\dot{\theta}_k(t) \propto e^{\theta^r_k - \theta_k(t)} - 1 = \frac{1}{e^{\theta_k(t)}} \left( e^{\theta^r_k} - e^{\theta_k(t)}\right),
\]
where the constant of proportionality depends on $\theta(t)$ but is independent of $k$. It follows that
\begin{equation} \label{eqn:flow.derivative}
\frac{d}{dt} e^{\theta_k(t)} \propto e^{\theta^r_k} - e^{\theta_k(t)}, \quad k = 1, \ldots, n - 1.
\end{equation}
Comparing \eqref{eqn:flow.derivative} and \eqref{eqn:primal.geodesic.candidate} we see that the primal flow is a time change of the primal geodesic.
\end{proof}

\subsection{Generalized Pythagorean theorem}
Having characterized the primal and dual geodesics, we are ready to prove the generalized Pythagorean theorem. Our proof makes use of the Riemannian gradients given in Lemma \ref{lem:gradient}. The reason is that these gradients appear to have the correct scaling which is easier to handle, as can be seen in the proof (see \eqref{eqn:productuv}).

\begin{proof}[Proof of Theorem \ref{thm:pyth}]

Given $p, q, r \in \Delta_n$, consider the primal geodesic from $q$ to $r$ and the dual geodesic from $q$ to $p$. Let
\begin{equation} \label{eqn:XY}
u = -\grad \ T\left(\cdot \mid p\right)(q) \quad \text{and} \quad v = -\grad \ T\left(r \mid \cdot\right)(q).
\end{equation}
By Theorem \ref{thm:gradient.flows} $u$ and $v$ are proportional to the initial velocities of the two geodesics. Thus, it suffices to prove that the sign of \eqref{eqn:pyth} is the same as that of $\langle u, v \rangle$. This claim will be established by the following two lemmas.

\begin{lemma} \label{lemma:pqr.lemma}
The sign of $T\left(q \mid p\right) + T\left(r \mid q\right) - T\left(r \mid p\right)$ is the same as that of
\begin{equation} \label{eqn:ortho.condition}
1 - \sum_{k = 1}^n \frac{\Pi_k(q, p) \Pi_k(r, q)}{\ppi_k(q)}.
\end{equation}
\end{lemma}
\begin{proof}
By Lemma \ref{lem:coord.rep}, the sign of  $T\left(q \mid p\right) + T\left(r \mid q\right) - T\left(r \mid p\right)$ is the same as that of
\[
\left( \sum_{i = 1}^n \ppi_i(\theta^p) e^{\theta_i^q - \theta_i^p} \right) \left(\sum_{j = 1}^n \ppi_j(\theta^q) e^{\theta_j^r - \theta_j^q}\right) - \sum_{i = 1}^n \ppi_i(\theta^p) e^{\theta_i^r - \theta_i^p}.
\]
Rearranging, we have
\[
-\sum_{i, j = 1}^n \ppi_i(\theta^p) (\delta_{ij} - \ppi_j(\theta^q)) e^{\theta_j^r - \theta_j^q} e^{\theta_i^q - \theta_i^p}.
\]
Since scaling does not change sign, we may consider instead the quantity
\[
-\sum_{i, j = 1}^n \ppi_i(\theta^p) (\delta_{ij} - \ppi_j(\theta^q)) \frac{\Pi_j(\theta^r, \theta^q)}{\ppi_j(\theta^q)}  \frac{\Pi_i(\theta^q, \theta^p)}{\ppi_i(\theta^p)}.
\]
We get \eqref{eqn:ortho.condition} by expanding.
\end{proof}

\begin{lemma}
Consider the tangent vectors $u$ and $v$ defined by \eqref{eqn:XY}. Then
\begin{equation} \label{eqn:productuv}
\langle u, v \rangle = 1 - \sum_{k = 1}^n \frac{\Pi_k(\theta^q, \theta^p) \Pi_k(\theta^r, \theta^q)}{\ppi_k(\theta^q)}.
\end{equation}
\end{lemma}
\begin{proof}
For this computation we use the primal coordinate system. We have
\begin{equation*}
\begin{split}
u &= -\grad \ T\left(\cdot \mid p\right)(q) =-\sum_{i, k = 1}^{n-1} g^{ik}(q) \frac{\partial}{\partial \theta_k^q} T\left( \cdot \mid p \right)(\theta^q) \frac{\partial }{\partial \theta_i^q}\\
v &= -\grad \ T\left(r \mid \cdot\right)(q) = -\sum_{j, \ell = 1}^{n-1} g^{j\ell}(q) \frac{\partial}{\partial \theta_{\ell}^q} T\left( r \mid \cdot \right)(\theta^q) \frac{\partial }{\partial \theta_j^q}.
\end{split}
\end{equation*}
Using the definition of the Riemannian inner product, we compute
\begin{equation*}
\begin{split}
\langle u, v \rangle &= \sum_{i, j = 1}^{n-1} g_{ij}(q) \sum_{k, \ell = 1}^{n-1} g^{ik}(q) g^{j\ell}(q) \frac{\partial}{\partial \theta_k^q} T\left( \cdot \mid p \right)(\theta^q) \frac{\partial}{\partial \theta_{\ell}^q} T\left( r \mid \cdot \right)(\theta^q) \\
  &= \sum_{k, \ell = 1}^{n-1} g^{k\ell}(\theta^q) \frac{\partial}{\partial \theta_k^q} T\left( \cdot \mid p \right)(\theta^q) \frac{\partial}{\partial \theta_{\ell}^q} T\left( r \mid \cdot \right)(\theta^q).
\end{split}
\end{equation*}
By \eqref{eqn:g.primal.inverse}, \eqref{eqn:T.derivative.phi.2} and \eqref{eqn:T.derivative.alternative}, we have
\begin{eqnarray*}
g^{k\ell}(q) &=& \frac{1}{\ppi_{k}(q)} \frac{\partial \theta_{\ell}}{\partial \phi_k}(\phi^q) + \frac{1}{\ppi_n(\theta^q)} \sum_{\alpha = 1}^{n-1} \frac{\partial \theta_{\ell}}{\partial \phi_\alpha}(\phi^q), \\
\frac{\partial}{\partial \theta_k^q} T\left( \cdot \mid p \right)(\theta^q) &=& \Pi_k(\theta^q, \theta^p) - \ppi_k(\theta^q), \\
\frac{\partial}{\partial \theta_{\ell}^q} T\left( r \mid \cdot \right)(\theta^q) &=& \sum_{\beta = 1}^{n - 1} \frac{\partial \phi_{\beta}}{\partial \theta_{\ell}^q}(\theta^q) \left(\ppi_{\beta}(\theta^q) - \Pi_{\beta}(\theta^r, \theta^q)\right).
\end{eqnarray*}

\medskip 
\noindent
{\it Claim.} We have
\begin{equation} \label{eqn:computation.claim}
\langle u, v \rangle = \sum_{k, \ell = 1}^{n-1} (\Pi_k(\theta^q, \theta^p) - \ppi_k(\theta^q))(\ppi_{\ell}(\theta^q) - \Pi_{\ell}(\theta^r, \theta^q))\left( \frac{\delta_{k\ell}}{\ppi_k(\theta^q)} + \frac{1}{\ppi_n(\theta^q)}\right).
\end{equation}

To see this, write
\begin{equation*}
\begin{split}
\langle u, v \rangle &= \sum_{k, \beta = 1}^{n-1} (\Pi_k(\theta^q, \theta^p) - \ppi_k(\theta^q))(\ppi_{\beta}(\theta^q) - \Pi_{\beta}(\theta^r, \theta^q)) \\
 &\quad \quad \cdot \sum_{\ell = 1}^{n-1} \left( \frac{1}{\ppi_k(q)} \frac{\partial \phi_{\beta}}{\partial \theta_{\ell}^q}(\theta^q) \frac{\partial \theta_{\ell}}{\partial \theta_k}(\phi^q) + \frac{1}{\ppi_n(q)} \sum_{\alpha = 1}^{n-1} \frac{\partial \phi_{\beta}}{\partial \theta_{\ell}^q}(\theta^q) \frac{\partial \theta_{\ell}}{\partial \phi_{\alpha}}(\phi^q)\right).
\end{split}
\end{equation*}
The last expression can be simplified using the identities
\[
\frac{\partial \phi_{\beta}}{\partial \theta_{\ell}^q}(\theta^q) \frac{\partial \theta_{\ell}}{\partial \theta_k}(\phi^q) = \delta_{\beta k}, \quad \frac{\partial \phi_{\beta}}{\partial \theta_{\ell}^q}(\theta^q) \frac{\partial \theta_{\ell}}{\partial \phi_{\alpha}}(\phi^q) = \delta_{\alpha \beta},
\]
and this gives the claim.

Finally, expanding and simplifying \eqref{eqn:computation.claim}, we obtain the desired identity \eqref{eqn:productuv}.
\end{proof}
With these two lemmas the proof is complete.
\end{proof}

\subsection{Application to mathematical finance} \label{sec:application.fgp} 
In this subsection we explain how the new information geometry can be applied to finance. For financial background and further details we refer the reader to \cite{PW14}. Consider sequential investment in a stock market with $n$ stocks. At time $t = 0, 1, 2, \ldots$, let $\mu(t) \in \Delta_n$ be the distribution of capital in the market. Explicitly, let $X_i(t) > 0$ be the market capitalization of stock $i$. Then
\begin{equation} \label{eqn:market.weight}
\mu_i(t) = \frac{X_i(t)}{X_1(t) + \cdots + X_n(t)}
\end{equation}
is the market weight of stock $i$.

Let $\ppi: \Delta_n \rightarrow \Delta_n$ be a portfolio map. It defines a self-financing investment strategy such that if the current state of the market is $\mu(t) = p$, we distribute our capital among the stocks according to the vector $\ppi(p)$. For example, if $\ppi(p) = \left(0.4, 0.5, 0.1\right)$, we invest $40\%$ in stock $1$, $50\%$ in stock $2$, and the rest in stock $3$. After each period, we perform the necessary trading such that the proportions of capitals invested are rebalanced to $\ppi(\mu(t + 1))$. If one chooses the identity map $\ppi(p) \equiv p$, the resulting portfolio is called the market portfolio which we denote by $\mu$. The market portfolio plays the role of benchmark. 

Now let $\ppi$ be the portfolio map generated by an exponentially concave function $\varphi$ (Definition \ref{def:fgp}), and consider the value of the resulting portfolio beginning with $\$1$ at time $0$. Also consider the value of the market portfolio starting with $\$1$. Consider the ratio
\[
V(t) := \frac{\text{value of the portfolio $\ppi$ at time $t$}}{\text{value of the portfolio $\mu$ at time $t$}}
\]
called the relative value of the portfolio $\ppi$. Under suitable conditions (see \cite{PW13, PW14}) it can be shown that $V(0) = 1$ and
\[
\frac{V(t + 1)}{V(t)} = \sum_{i = 1}^n \ppi_i(\mu(t)) \frac{\mu_i(t + 1)}{\mu_i(t)}.
\]

\begin{proposition} [Fernholz's decomposition]  \cite{F99, PW14}
For all time $t$ we have
\[
\log V(t) = (\varphi(\mu(t)) - \varphi(\mu(0))) + \sum_{s = 0}^{t - 1} T\left(\mu(s + 1) \mid \mu(s)\right).
\]
\end{proposition}

In this sense, the cumulative L-divergence measures the volatility harvested by the portfolio $\ppi$. See \cite{W15} for optimization of this quantity.

\begin{figure}[t!]
\centering
\makebox[\textwidth][c]{\includegraphics[scale=0.4]{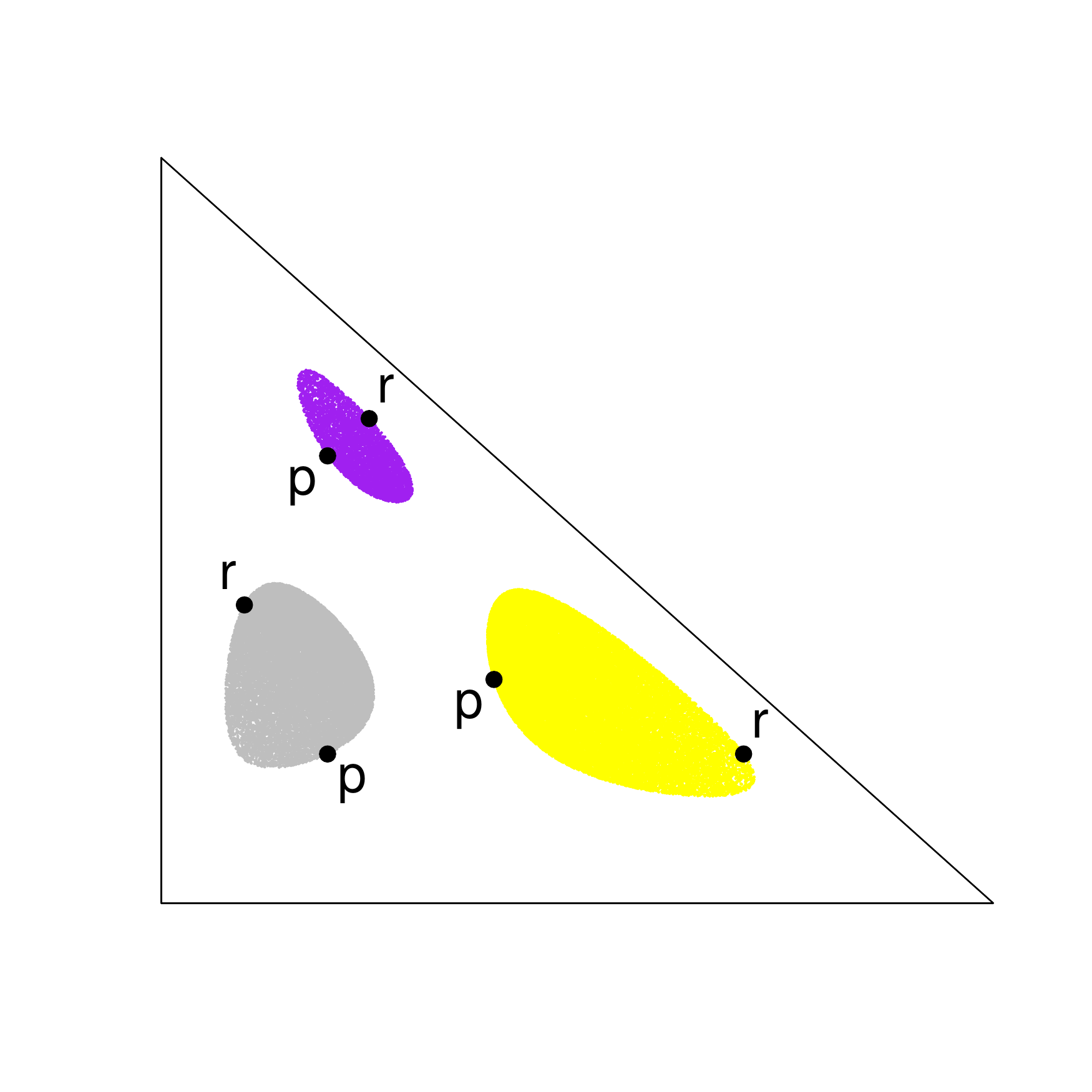}}
\vspace{-1cm}
\caption{Plots of the region $\{q \in \Delta_n: T\left(q \mid p\right) + T\left(r \mid q\right) \leq T\left(r \mid p\right)\}$ for the equal-weighted portfolio $\ppi \equiv \left(\frac{1}{3}, \frac{1}{3}, \frac{1}{3}\right)$, for several values of $p$ and $r$. Each point $q$ on the boundary gives a `right-angled geodesic triangle' in the sense of Theorem \ref{thm:pyth}.} \label{fig:region}
\end{figure}

The above discussion assumes that the portfolio rebalances every period (say every week). In practice, due to transaction costs and other considerations, we may want to rebalance at other frequencies. Let $0 = t_0 < t_1 < t_2$ be three time points and consider two ways of implementing the portfolio $\ppi$: (i) rebalance at times $t_0$ and $t_1$ (ii) rebalance at time $t_0$ only. By Fernholz's decomposition, the relative values of the two implementations at time $t_2$ are
\begin{eqnarray*}
\log V^{(1)}(t_2) &=& (\varphi(\mu(t_2)) - \varphi(\mu(t_0))) + T\left(\mu(t_1) \mid \mu(t_0)\right) + T\left(\mu(t_2) \mid \mu(t_1)\right), \\
\log V^{(2)}(t_2) &=& (\varphi(\mu(t_2)) - \varphi(\mu(t_0))) + T\left(\mu(t_2) \mid \mu(t_0)\right). 
\end{eqnarray*}
Letting $\mu(t_0) = p$, $\mu(t_1) = q$ and $\mu(t_2) = r$, the difference between the two values is
\[
\log V^{(1)}(t_2) - \log V^{(2)}(t_2) = T\left(q \mid p\right) + T\left(r \mid q\right) - T\left(r \mid p\right).
\]

By the generalized Pythagorean theorem, the sign of the difference is determined by the angle between the dual geodesic from $q$ to $p$ and the primal geodesic from $q$ to $r$. In Figure \ref{fig:region} we illustrate the result for the equal-weighted portfolio of $3$ stocks. In the figure, rebalancing at time $t_1$ creates extra profit if and only if $q$ lies outside the region. This shows convincingly that rebalancing more frequently is not always better, even in the absence of transaction costs. More importantly, our framework provides a geometric and model-independent way of saying that the sequence $(p, q, r)$ is `more volatile' than the subsequence $(p, r)$.

\section{Displacement interpolation} \label{sec:interpolation}
In this final section we consider displacement interpolation for the optimal transport problem formulated in Section \ref{sec:the.transport.problem}. We refer the reader to \cite[Chapter 5]{V03} and \cite[Chapter 7]{V08} for introductions to displacement interpolation. 

\subsection{Time dependent transport problem}
Let $P^{(0)}$ and $P^{(1)}$ be Borel probability measures on ${\mathbb{R}}^{n-1}$. Consider the transport problem with cost $c$ given by \eqref{eqn:psi.intro}. Suppose the transport problem is solved in terms of the exponentially concave function $\varphi$ on $\Delta_n$. Letting $f = \varphi + \psi$, the (Monge) optimal transport map is given by the $c$-supergradient of $f$. In particular, $P^{(1)}$ is the pushforward of $P^{(0)}$ under $F := \nabla^c f$:
\[
P^{(1)} = F_{\#} P^{(0)}.
\]

The idea of displacement interpolation is to introduce an additional time structure. We want to define an `action' ${\mathcal{A}}(\cdot)$ on curves such that the cost function is given by
\[
c(\theta, \phi) = \min_{\gamma} {\mathcal{A}}(\gamma),
\]
where the minimum is taken over smooth curves $\gamma: [0, 1] \rightarrow {\mathbb{R}}^{n-1}$ satisfying $\gamma(0) = \theta$ and $\gamma(1) = \phi$. For each pair $(\theta, \phi)$, a minimizing curve $\gamma$ gives a time-dependent map transporting $\theta$ to $\phi$ along a continuous path. Let $F^{(t)}: {\mathbb{R}}^{n-1} \rightarrow {\mathbb{R}}^{n-1}$ be defined by $F^{(t)}(\theta) = \gamma(t)$, where $\gamma$ is the minimizing curve for the pair $(\theta, F(\theta))$. We want to define ${\mathcal{A}}$ in such a way that $F^{(t)}$ is an optimal transport map for the probability measures $P^{(0)}$ and $P^{(t)}$ where
\[
P^{(t)} = \left(F^{(t)}\right)_{\#} P_0.
\]

For the classical Euclidean case with cost $|x - y|^2$ the action is ${\mathcal{A}}(\gamma) = \int_0^1 |\dot{\gamma}(t)|^2 dt$ and the optimal transport map has the form $F(x) = x - \nabla h(x)$ where $h$ is an ordinary concave function. The displacement interpolations are linear interpolations:
\[
F^{(t)} = (1 - t) \mathrm{Id} + t F
\]
(See \cite[Theorem 5.5, Theorem 5.6]{V03}.) In particular, the individual trajectories (minimizing curves) are Euclidean straight lines which can be regarded as the geodesics of a flat geometry. In this section we formulate and prove an analogous statement for our transport problem.

\subsection{Lagrangian action and portfolio interpolation}
We begin by defining an appropriate action. Let $\gamma: [0, 1] \rightarrow {\mathbb{R}}^{n-1}$ be a smooth curve with $\gamma(0) = \theta$. For each $t$, define $q(t) \in \Delta_n$ such that its exponential coordinate is $\theta - \gamma(t)$, i.e.,
\begin{equation} \label{eqn:q.map}  
\frac{q_i(t)}{q_n(t)} = e^{\theta_i - \gamma_i(t)}, \quad 1 \leq i \leq n - 1.
\end{equation}
Equivalently, we have $q_i(t) = e^{\theta_i - \gamma_i(t) - \psi(\theta - \gamma(t))}$, for $1 \leq i \leq n - 1$. Intuitively, we think of $q(t)$ as the portfolio at time $t$ (in the sense of interpolation). Note that $q(0) = \left(\frac{1}{n}, \ldots, \frac{1}{n}\right)$.

We define the Lagrangian action by 
\begin{equation}  \label{eqn:action}
{\mathcal{A}}(\gamma) = \int_0^1 -\log\left(\frac{1}{n} + \dot{q}_n(t)\right) dt.
\end{equation}
We take $-\log (\cdot) = \infty$ if the argument is not in $(0, \infty)$.  An alternative representation of the action is
\begin{equation} \label{eqn:action2}
{\mathcal{A}}(\gamma) = \int_0^1 -\log\left(\frac{1}{n} + \frac{d}{dt} e^{-\psi(\gamma(0) - \gamma(t))} \right) dt.
\end{equation}

\begin{lemma}
For any $\theta, \phi \in {\mathbb{R}}^{n-1}$ we have
\begin{equation} \label{eqn:cost.as.min}
c(\theta, \phi) = \psi(\theta - \phi) = \min \left\{ {\mathcal{A}}(\gamma): \gamma(0) = \theta, \gamma(1) = \phi\right\}.
\end{equation}
The action is minimized by the curve
\begin{equation} \label{eqn:minimizing.curve}
\gamma_i(t) = \theta_i - \log \frac{(1 - t) \frac{1}{n} + t q_i(1)}{(1 - t) \frac{1}{n} + t q_n(1)}, \quad 1 \leq i \leq n - 1.
\end{equation}
In particular, for this minimizing curve we have
\begin{equation} \label{eqn:portfolio.interpolation}
q(t) = (1 - t)\left(\frac{1}{n}, \ldots, \frac{1}{n}\right) + t q(1).
\end{equation}
\end{lemma}
\begin{proof}
Fix a smooth curve $\gamma: [0, 1] \rightarrow {\mathbb{R}}^{n-1}$ from $\theta$ to $\phi$. Since $-\log$ is convex, by Jensen's inequality we have
\begin{equation} \label{eqn:jensen}
\begin{split}
\int_0^1 &-\log \left( \frac{1}{n} + \dot{q}_n(t)\right) dt \geq -\log \left(\frac{1}{n} + \int_0^1 \dot{q}_n(t) dt \right) \\
&= -\log \left(\frac{1}{n} + q_n(1) - q_n(0)\right) = -\log q_n(1) = \psi(\theta - \phi).
\end{split}
\end{equation}
For the curve defined by \eqref{eqn:minimizing.curve}, $\dot{q}(t) = q(1) - \frac{1}{n}$ is constant and so equality holds in \eqref{eqn:jensen}. Finally \eqref{eqn:portfolio.interpolation} follows by a direct calculation.
\end{proof}

\subsection{Displacement interpolation} \label{sec:interpolation.main}
We work under the following setting. Let $P^{(0)}$ and $P^{(1)}$ be Borel probability measures on ${\mathbb{R}}^{n-1}$. Let $\varphi: \Delta_n \rightarrow {\mathbb{R}}$ be an exponentially concave function, satisfying Assumption \ref{ass:regularity}, such that $F^{(1)} := \nabla^c f$ is an optimal transport map (here $f$ is the $c$-concave function $\varphi + \psi$). Let $\ppi^{(1)}: \Delta_n \rightarrow \Delta_n$ be the portfolio map generated by $\varphi^{(1)} := \varphi$.

Consider the flow $(t, \theta) \mapsto \phi^{(t)}(\theta)$ defined by the minimizing curves \eqref{eqn:minimizing.curve}, i.e.,
\begin{equation} \label{eqn:flow}
\phi_i^{(t)}(\theta) = \theta_i - \log \frac{\ppi^{(t)}_i(\theta)}{\ppi^{(t)}_n(\theta)}, \quad 1 \leq i \leq n - 1, \quad t \in [0, 1],
\end{equation}
where each $\ppi^{(t)} : \Delta_n \rightarrow \Delta_n$ is the portfolio map defined by
\begin{equation} \label{eqn:pi.flow}
\ppi^{(t)} := (1 - t) \left(\frac{1}{n}, \ldots, \frac{1}{n}\right) + t \ppi^{(1)}, \quad t \in [0, 1].
\end{equation}
The following is the main result of this section. It is interesting to note that the displacement interpolation can be interpreted naturally as the linear interpolation between the equal-weighted portfolio and the terminal portfolio.

\begin{theorem}[Displacement interpolation] \label{thm:displacement.interpolation} {\ } 
Consider the setting of Section \ref{sec:interpolation.main}.
\begin{enumerate}
\item[(i)] For each $t \in [0, 1]$, the portfolio map $\ppi^{(t)}$ is generated by the exponentially concave function $\varphi^{(t)}$ on $\Delta_n$ defined by
\begin{equation}
\varphi^{(t)}(p) = (1 - t) \sum_{i = 1}^n \frac{1}{n} \log p_i + t \varphi(p), \quad p \in \Delta_n.
\end{equation}
\item[(ii)] For each $t \in [0, 1]$, let $f^{(t)} = \varphi^{(t)} + \psi$ and let $F^{(t)} = \nabla^c f^{(t)}$. If $\theta$ is distributed according to $P^{(0)}$, then $\theta^{(t)}$ is distributed according to $P^{(t)}$ where
\[
P^{(t)} = \left(F^{(t)}\right)_{\#} P^{(0)}.
\]
Moreover, $F^{(t)}$ is an optimal transport map for the transport problem for $(P^{(0)}, P^{(t)})$.
\item[(iii)] Endow $\Delta_n$ with the geometric structure induced by the L-divergence of $\varphi$. We further assume that the $c$-gradient $F^{(1)} = \nabla^c f: {\mathbb{R}}^{n-1} \rightarrow {\mathbb{R}}^{n-1}$ is surjective. For each $\theta \in {\mathbb{R}}^{n-1}$ fixed, consider the curve $t \mapsto \phi^{(t)}(\theta)$ in dual coordinates. Then the trace of the curve is the dual geodesic joining $\theta$ and $\phi^{(1)}(\theta)$.
\end{enumerate}
\end{theorem}
\begin{proof}
(i) Follows directly from Example \ref{ex:fgp.example}(iv).

(ii) It is clear that $(\theta, F^{(t)}(\theta))$ is a coupling of $(P^{(0)}, P^{(t)})$. By Proposition \ref{prop:fgp.mcm}, the graph of the map $F^{(t)}$ is $c$-cyclical monotone. This proves that $F^{(t)}$ is an optimal transport map.

(iii) We write \eqref{eqn:flow} in the form
\begin{equation*}
\begin{split}
e^{-\phi_i^{(t)}} &= e^{-\theta_i} \frac{(1 - t)\frac{1}{n} + t\ppi_i(\theta)}{(1 - t)\frac{1}{n} + t\ppi_n(\theta)} \\
  &= \frac{(1 - t) \frac{1}{n}}{(1 - t)\frac{1}{n} + t \ppi_n(\theta)} e^{-\theta_i} + \frac{t\ppi_n(\theta)}{(1 - t)\frac{1}{n} + t \ppi_n(\theta)} e^{-\theta_i + \log \frac{\ppi_i(\theta)}{\ppi_n(\theta)}} \\
  &=: (1 - h(t)) e^{-\theta_i} + h(t) e^{-\phi^{(1)}_i(\theta)}.
\end{split}
\end{equation*}
By \eqref{eqn:dual.geodesic.curve} we see that $t \mapsto \phi^{(t)}$ is a time change of a dual geodesic. The surjectivity assumption guarantees that the curve lies within ${\mathcal{Y}}'$, the range of the dual coordinate system.
\end{proof}

\subsection{Another interpolation}
From the financial perspective there is another natural interpolation, namely the linear interpolation between the market portfolio $\mu$ (Example \ref{ex:fgp.example}(i)) and the portfolio $\ppi$:
\begin{equation} \label{eqn:market.interpolation}
\ppi^{(t)} = (1 - t) \ppi + t \mu.
\end{equation}
The corresponding log generating function is $\varphi^{(t)} = (1 - t) \varphi$. From the transport perspective, the market portfolio corresponds to the trivial transport map $F(\theta) \equiv 0$ (recall in Proposition \ref{prop:fgp.mcm}(iii) that the portfolio has exponential coordinate given by $\theta - F(\theta)$). By the argument of Theorem \ref{thm:displacement.interpolation} we have the following result.

\begin{proposition}
Consider the geometric structure induced by $\varphi$ and assume that the range of the dual coordinate system is ${\mathbb{R}}^{n-1}$. Consider the flow $(t, \theta) \mapsto \phi^{(t)}(\theta)$ in \eqref{eqn:flow} where $\ppi^{(t)}$ is given by the interpolation \eqref{eqn:market.interpolation}. Then for each $\theta$, in dual coordinates, the trace of the curve $t \mapsto \phi^{(t)}(\theta)$ is a time change of the dual geodesic from $\phi^{(0)}(\theta)$ to $0$.
\end{proposition}

\appendix

\section{Technical proofs}\label{sec:appendix}

\begin{proof}[Proof of Theorem \ref{thm:duality}]
In this proof we treat $\theta$ and $\phi$ as independent variables.

We prove (i) and (ii) together. We begin by observing that
\begin{equation} \label{eqn:weight.as.derivative}
\frac{\partial}{\partial \theta_i} f(\theta) = \ppi_i(\theta), \quad 1 \leq i \leq n - 1.
\end{equation}
To see this, write $p_i = e^{\theta_i - \psi(\theta)}$. Switching coordinates and using the chain rule, we have
\begin{equation*}
\begin{split}
\frac{\partial}{\partial \theta_i} f(\theta) &= \sum_{j = 1}^n \frac{\partial \varphi}{\partial p_j}(p)  (-p_ip_j) + \frac{\partial \varphi}{\partial p_i}(p)p_i + p_i \\
&= p_i \left(1 + \frac{\partial \varphi}{\partial p_i}(p) - \sum_{j = 1}^n p_j \frac{\partial \varphi}{\partial p_j}(p)\right).
\end{split}
\end{equation*}
From \eqref{eqn:fgp.weights} and a bit of computation, we see that this equals $\ppi_i(\theta)$.

Consider the $c$-transform of $f$ given by
\begin{equation} \label{eqn:c.transform.f.1}
f^*(\phi) = \inf_{\theta \in {\mathcal{X}}} \left(\psi(\theta - \phi) - f(\theta)\right).
\end{equation}

Differentiating $\psi(\theta - \phi) - f(\theta)$ and using \eqref{eqn:weight.as.derivative}, we see that $\theta \in {\mathcal{X}}$ attains the infimum in \eqref{eqn:c.transform.f.1} if and only if
\[
\frac{e^{\theta_i - \phi_i}}{\sum_{j = 1}^n e^{\theta_j - \phi_j}} = \ppi_i(\theta), \quad i = 1, \ldots, n - 1.
\]
Rearranging, we have
\begin{equation} \label{eqn:theta.dual}
\phi_i = \theta_i - \log \frac{\ppi_i(\theta)}{\ppi_n(\theta)}, \quad i = 1, \ldots, n - 1.
\end{equation}
This proves that equality holds in
\begin{equation} \label{eqn:fenchel.ineq}
f(\theta) + f^*(\phi) \leq \psi(\theta - \phi)
\end{equation}
if and only if $\theta$ and $\phi$ satisfies the relation \eqref{eqn:theta.dual}. In particular, for $\theta \in {\mathcal{X}}$ the $c$-supergradient $\nabla^c f(\theta)$ is given by \eqref{eqn:theta.dual}.

Next we we prove that the minimizer in \eqref{eqn:c.transform.f.1}, if exists, is unique. Consider instead maximization of the quantity
\[
e^{f(\theta) - \psi(\theta - \phi)} = e^{\varphi(\theta) + \psi(\theta) - \psi(\theta - \phi)} = \Phi(\theta) \frac{e^{\psi(\theta)}}{e^{\psi(\theta - \phi)}}.
\]
Expanding and switching to Euclidean coordinates, this equals
\begin{equation} \label{eqn:quotient}
\Phi(p) \frac{\sum_{i = 1}^n e^{\theta_i}}{\sum_{i = 1}^n e^{\theta_i - \phi_i}} = \Phi(p) \frac{1}{\sum_{i = 1}^n a_i p_i},
\end{equation}
where $a_i = e^{-\phi_i} > 0$. Being the quotient of a strictly concave function and an affine function, the right hand side of \eqref{eqn:quotient} is strictly quasi-concave, i.e., its superlevel sets are strictly convex (see \cite[Example 3.38]{BV04}). This shows that the minimizer $\theta$ in \eqref{eqn:c.transform.f.1} is unique if it exists.

Let $\phi \in {\mathcal{Y}}'$. Then there exists unique $\theta \in {\mathcal{X}}$ such that equality holds in \eqref{eqn:fenchel.ineq} and $\phi = \nabla^c f(\theta)$. In particular, the $c$-supergradient $\partial^c f^*(\phi)$ is $\theta$ and $\nabla^c f^* (\nabla^c f(\theta)) = \theta$. This completes the proof of (i) and (ii).

Next we prove that $\nabla^c f^* = (\nabla^c f)^{-1}$ is smooth, so that $\nabla^c f: {\mathcal{X}} \rightarrow {\mathcal{Y}}'$ is a diffeomorphism. Since $\nabla^c f: {\mathcal{X}} \rightarrow {\mathcal{Y}}$ is smooth and injective, by the inverse function theorem it suffices to show that the Jacobian of $\nabla^c f$ is invertible on ${\mathcal{X}}$. Invertibility of the Jacobian follows from \eqref{eqn:g.matrix} and the fact that $\left(g_{ij}(\xi)\right)$ is strictly positive definite.

Finally, by Fenchel's identity we can express $f^*$ in the form
\[
f^*(\phi) = c(\nabla^c f^*(\phi), \phi) - f(\nabla^c f^*(\phi)).
\]
Since $\nabla^c f^*$ and the cost function $c$ are smooth, we see that $f^*$ is smooth as well.
\end{proof}

\begin{proof}[Proof of Lemma \ref{lem:gradient}]
(i) To prove the first formula in \eqref{eqn:Riemannian.gradient.primal}, we compute, using \eqref{eqn:g.primal.inverse} and \eqref{eqn:T.partial.i},
\begin{equation*}
\begin{split}
& \left(\grad \ T\left( \cdot \mid \theta^p\right)(\theta^q)\right)_i\\
 &=  \sum_{j = 1}^{n - 1} g^{ij}(\theta^q) \frac{\partial}{\partial \theta_j^q} T\left(\cdot \mid \theta^p\right)(\theta^q) \\
  &= \sum_{j = 1}^{n-1} \left(\frac{1}{\ppi_j(\theta^q)} \frac{\partial \theta_i}{\partial \phi_j}(\theta^q) + \frac{1}{\ppi_n(\theta^q)} \sum_{k = 1}^{n-1} \frac{\partial \theta_i}{\partial \phi_k}(\theta^q)\right) \left( \Pi_j(\theta^q, \theta^p) - \ppi_j(\theta^q)\right) \\
  &= \sum_{j = 1}^{n - 1} \left( \frac{\Pi_j(\theta^q, \theta^p)}{\ppi_j(\theta^q)} - \frac{\Pi_n(\theta^q, \theta^p)}{\ppi_n(\theta^q)}\right) \frac{\partial \theta_i}{\partial \phi_j}(\theta^q).
\end{split}
\end{equation*}

For the second formula, we first prove a

\medskip

\noindent
{\it Claim.} We have
\begin{equation} \label{eqn:T.derivative.alternative}
\frac{\partial}{\partial \theta_j^q} T \left( \theta^r \mid \theta^q\right) = \sum_{\ell = 1}^{n-1} \frac{\partial \phi_{\ell}}{\partial \theta_j^q}(\theta^q)(\ppi_{\ell}(\theta^q) - \Pi_{\ell}(\theta^r, \theta^q)).
\end{equation}
To see this, we use \eqref{eqn:T.partial.i.prime}, \eqref{eqn:pi.derivative}, and compute as follows:
\begin{equation*}
\begin{split}
& \frac{\partial}{\partial \theta_j^q} T\left(\theta^r \mid \theta^q\right) \\
&= -\Pi_j(\theta^r, \theta^q) + \ppi_j(\theta^q) + \sum_{\ell = 1}^n \frac{1}{\ppi_{\ell}(\theta^q)} \frac{\partial \ppi_{\ell}}{\partial \theta_j^q}(\theta^q) \Pi_{\ell}(\theta^r, \theta^q) \\
&= -\Pi_j(\theta^r, \theta^q) + \ppi_j(\theta^q) + \\
&\quad \quad \sum_{\ell = 1}^n \left(\delta_{\ell j} - \ppi_j(\theta^q) - \left(\frac{\partial \phi_{\ell}}{\partial \theta_j^q}(\theta^q) - \sum_{m = 1}^{n-1} \ppi_m(\theta^q) \frac{\partial \phi_m}{\partial \theta_j^q}(\theta^q)\right)\Pi_{\ell}(\theta^r, \theta^q)\right)\\
&= \sum_{\ell = 1}^{n-1} \frac{\partial \phi_{\ell}}{\partial \theta_j^q}(\theta^q)(\ppi_{\ell}(\theta^q) - \Pi_{\ell}(\theta^r, \theta^q)).
\end{split}
\end{equation*}

Now we compute, using Theorem \ref{thm:Riemannian.metric} and the symmetry of $g^{ij}$,
\begin{equation*}
\begin{split}
& \left( \grad \ T\left(r \mid \cdot\right)(q)\right)_i \\
&= \sum_{j = 1}^{n - 1} g^{ij}(\theta^q) \frac{\partial}{\partial \theta_j^q} T \left( \theta^r \mid \cdot\right)(\theta^q) \\
&= \sum_{j = 1}^{n-1} \left(\frac{1}{\ppi_i(\theta^q)} \frac{\partial \theta_j}{\partial \phi_i}(\phi^q) + \frac{1}{\ppi_n(\theta^q)} \sum_{k = 1}^{n-1} \frac{\partial \theta_j}{\partial \phi_k}(\phi^q)\right)  \cdot  \\
&\quad \quad \sum_{\ell = 1}^{n-1} \frac{\partial \phi_{\ell}}{\partial \theta_j^q}(\theta^q)(\ppi_{\ell}(\theta^q) - \Pi_{\ell}(\theta^r, \theta^q)) \\
&= \sum_{\ell = 1}^{n-1} (\ppi_{\ell}(\theta^q) - \Pi_{\ell}(\theta^r, \theta^q))  \cdot \\
&\quad \quad \sum_{j = 1}^{n-1} \left( \frac{1}{\ppi_i(\theta^q)} \frac{\partial \phi_{\ell}}{\partial \theta_j^q}(\phi^q) \frac{\partial \theta_j}{\partial \phi_i}(\phi^q) + \frac{1}{\ppi_n(\theta^q)} \sum_{k = 1}^{n-1} \frac{\partial \phi_{\ell}}{\partial \theta_j^q}(\phi^q) \frac{\partial \theta_j}{\partial \phi_k}(\phi^q) \right) \\
&= \sum_{\ell = 1}^{n-1} (\ppi_{\ell}(\theta^q) - \Pi_{\ell}(\theta^r, \theta^q))\left(\frac{1}{\ppi_i(\theta^q)} \delta_{\ell i} + \frac{1}{\ppi_n(\theta^q)}\right) \\
&= -\frac{\Pi_i(\theta^r, \theta^q)}{\ppi_i(\theta^q)} + \frac{\Pi_n(\theta^r, \theta^q)}{\ppi_n(\theta^q)}.
\end{split}
\end{equation*}
In the second last equality we used \eqref{eqn:jacobian.identity}. The proof of (ii) is similar.
\end{proof}

\section*{Acknowledgements}
The authors would like to thank Prof.~Robin Graham for helpful discussions. We also thank the anonymous referee for suggestions which improved the paper.

\bibliographystyle{imsart-nameyear}
\bibliography{rebalance}


\end{document}